\documentclass[10pt,a4paper]{article}
\usepackage{amsmath,amssymb,amsthm,amsfonts,mathtools,enumitem}
\usepackage{hyperref}
\usepackage[margin=1in]{geometry}

\def\y{\mathbf{y}}

\def\N{\mathbb{N}}
\def\R{\mathbb{R}}

\newtheorem{theorem}{Theorem}
\newtheorem{lemma}[theorem]{Lemma}
\newtheorem{proposition}[theorem]{Proposition}
\newtheorem{corollary}[theorem]{Corollary}
\theoremstyle{remark}
\newtheorem{remark}[theorem]{Remark}
\theoremstyle{definition}
\newtheorem{definition}[theorem]{Definition}

\newcommand{\multiindex}{I}
\newcommand{\polys}[1]{\R[x]_{#1}} %
\newcommand{\polysvar}[2]{\R[x_{#1}]_{#2}} %

\newcommand{\inner}[3]{\langle #1 , #2 \rangle_{#3}} %
\newcommand{\kernel}[2]{K_{#1}^{#2}} %
\newcommand{\jkernel}[1]{{\kernel{#1}{\mathrm{Jac}}}}

\newcommand{\chebyshevpoly}{T}
\newcommand{\kernelop}[2]{{\mathbf K_{#1}^{#2}}}
\newcommand{\jkernelop}[1]{\kernelop{#1}{\mathrm{Jac}}}
\newcommand{\bfr}{{\mathbf{r}}}
\newcommand{\J}{J}
\newcommand{\ones}{\mathbf{1}}
\newcommand{\twos}{\mathbf{2}}

\newcommand{\chebyshevmeas}[1]{\mu_{#1}}
\newcommand{\hamming}{w}
\newcommand{\Nzero}{\N_0}
\newcommand{\preordering}[2]{\mathcal P_{#1,#2}(\{1-x_i^2\}_{i\in #2})}
\newcommand{\preorderingg}[3]{\mathcal P_{#1,#3}(#2)}

\newcommand{\quadraticmoduleg}[3]{\mathcal Q_{#1,#3}(#2)}
\newcommand{\indexset}[1]{{\mathcal I_{#1}}}
\newcommand{\indexsetp}{{\indexset{p}}}
\newcommand{\sos}[1]{{\Sigma[x_{#1}]}}
\newcommand{\constraintg}{g}
\DeclareMathOperator{\fulldeg}{\overline\deg}
\DeclareMathOperator{\lip}{Lip}
\newcommand{\interset}{\mathcal{J}}

\newcommand{\maxJ}{{\overline{J}}}

\newcommand{\maxlip}{{\overline{L}}}
\newcommand{\maxdeg}{M}

\newcommand{\inter}[2]{#1_{#2}}
\newcommand{\Cjackson}{C_{\mathrm{Jac}}}

\newcommand{\domain}{S(\mathbf g)}
\newcommand{\cardg}{{\bar k}}
\newcommand{\lojL}{{\mathsf L}}
\newcommand{\lojC}{{\mathsf c}}
\newcommand{\bmh}{{\mathsf h}}
\newcommand{\Cf}{{C_f}}
\newcommand{\Cm}{{C_m}}
\newcommand{\Cd}{{C_d}}
\newcommand{\Call}{{\mathbf C_{j}}}

\newcommand{\f}{{\hat f}}
 
\usepackage{xcolor}
\definecolor{dkgreen}{rgb}{0,0.6,0}
\definecolor{dkblue}{rgb}{0,0,0.6}
\newcommand{\MK}{}%
\newcommand{\RRZ}{}%

\title{Convergence rates for sums-of-squares hierarchies with correlative sparsity}
\author{Milan Korda \and Victor Magron \and Rodolfo R\'ios Zertuche}
\date{\today}

\begin{document}

\maketitle

\begin{abstract}
    {This work derives upper bounds on the convergence rate of the moment-sum-of-squares hierarchy with correlative sparsity  \MK{for global minimization of polynomials on compact basic semialgebraic sets}. The main conclusion is that both sparse hierarchies based on the Schm\"udgen and Putinar Positivstellens\"atze enjoy a polynomial rate of convergence that depends on the size of the largest clique in the sparsity graph but not on the ambient dimension. Interestingly, the sparse bounds outperform the best currently available bounds for the dense hierarchy when the maximum clique size is sufficiently small compared to the ambient dimension and the performance is measured by the running time of an interior point method required to obtain a bound on the global minimum of a given accuracy.}
\end{abstract}

\tableofcontents

\section{Introduction}
\label{sec:intro}

This work  provides rates of convergence for the sums-of-squares hierarchy with correlative sparsity. For a positive $n\in\N$, consider the polynomial optimization problem
\[f_{\min{}} \coloneqq\min_{x\in S(\mathbf g)}f(x)\]
where $f$ is an element of the ring $\R[x]$ of polynomials in $x=(x_1,\dots,x_n)$, and $S(\mathbf g)$ is a basic compact  semialgebraic set determined by a finite collection of polynomials $\mathbf g=\{g_1,\dots,g_{\cardg}\}$ by $S(\mathbf g)=\{x\in\R^n:g_i(x)\geq 0,\;i=1,\dots,\cardg\}$. An approach to attack this problem, first proposed by Lasserre \cite{lasserre2001global} and Parrilo \cite{parrilo2003semidefinite}, is as follows: Imagine we knew that $f(x)-\lambda$ could be written as 
\[f(x)-\lambda=\sum_{j=0}^{\bar k}\sigma_j g_j(x)\quad\text{or}\quad f(x)-\lambda=\sum_{J\subseteq \{1,\dots,\cardg\}} \sigma_J \prod_{j\in J}g_j(x),\] 
with $g_0(x)=1$ and $\sigma_j$ and $\sigma_J$ being sum-of-squares (SOS) polynomials. Then the right-hand sides of each of these equations would be clearly nonnegative on $S(\mathbf g)$, so we would know that $f_{\min}\geq \lambda$. By bounding the degree of the SOS polynomials, we obtain the following two hierarchies of lower bounds:
\begin{align*}
 \mathrm{lb}_q(f,r)&\textstyle=\max\{\lambda\in \R:f-\lambda=\sum_{j=0}^\cardg \sigma_j g_j, \\
 &\textstyle\qquad\qquad\deg( \sigma_j g_j)\leq 2r,\;\sigma_j\in \Sigma[x]\},\\
  \mathrm{lb}_p(f,r)&=\textstyle\max\{\lambda\in \R:f-\lambda= \sum_{J\subseteq \{1,\dots,\cardg\}} \sigma_J\prod_{j\in J}g_j,\\
  &\textstyle\qquad\qquad\deg \left(\sigma_J\prod_{j\in J}g_j\right)\leq 2r,\; \sigma_J\in\Sigma[x]\},
\end{align*}
where $\Sigma[x]$ is the convex cone of all sum-of-squares polynomials. These satisfy $\mathrm{lb}_q(f,r)\leq \mathrm{lb}_p(f,r)\leq f_{\min}$. The lower bound $ \mathrm{lb}_q(f,r)$ is associated to a so-called \emph{quadratic module certificate}, while $ \mathrm{lb}_p(f,r)$ corresponds to a \emph{preordering certificate}; this terminology is justified by the definitions in Section \ref{sec:result}.
The well-known Putinar and Schm\"udgen Positivstellens\"atze \cite{putinar1993positive,schmudgen1991thek}, respectively, guarantee that these bounds converge to $f_{\min}$ as $r\to+\infty$, the former with the additional assumption that the associated quadratic module be Archimedian\footnote{This means that there are $R>0$ and $\sigma_j\in\Sigma[x]$ such that $R-\|x\|^2=\sum_{j}\sigma_j g_j(x)$.}. \MK{Here we will prove  sparse quantitative versions of these results}.

Polynomial optimization schemes have generated substantial interest due to their abundant fields of application; see for example \cite{laurent2009sums,lasserre2009moments}. {The first proof of convergence, without a convergence rate,  was given by Lasserre \cite{lasserre2001global} using the Archimedian positivstellensatz due to Putinar \cite{putinar1993positive}. }
Eventually, rates of convergence were obtained; initially in \cite{nie2007complexity} these were logarithmic  in the degree of the polynomials involved, and later on they were improved \cite{fang2021sum,laurent2021effective,slot2022sum,baldi2021moment} (using ideas of \cite{reznick1995uniform,doherty2012convergence,parrilo2013approximation}) to polynomial rates; refer to Table  \ref{tab:lb}. The crux of the argument used to obtain those rates is a bound of the deformation incurred by a polynomial strictly-positive on the domain of interest, as it passes through an integral operator that closely approximates the identity and is associated to a strictly-positive polynomial kernel that is itself composed of  sums of squares and similar to the Christoffel-Darboux and Jackson kernels (see Definition \ref{eq:jacksonone}). 

The techniques used to obtain these results generally involve linear operators on the space of polynomials (mostly Christoffel-Darboux kernel operators; see \cite{slot2022sum}) that are close to the identity and that, for positive polynomials, are easily (usually, by construction) proved to output polynomials that are sums of squares and/or of their products with the functions in $\mathbf g$. All of these results deal, however, with the dense case.

\if{
A related approach was introduced by Lasserre \cite{lasserre2014new}, where the following upper bounds are proposed:
\begin{align*}
  \mathrm{ub}_{ms}(f,r)&\textstyle=\inf\{L(f)\mid L\colon\R[x]_{2r}\to\R\text{ is linear, $L(1)=1$, and}\\
    &\qquad \textstyle L(q)\geq0\;\forall q=\sum_{i}p_i(x)^2,\;p_i\in \R[x],\;\deg q\leq 2r\}.\\
  \mathrm{ub}_{mq}(f,r)&\textstyle=\inf\{L(f)\mid L\colon\R[x]_{2r}\to\R\text{ is linear, $L(1)=1$, and}\\
    &\qquad \textstyle L(q)\geq0\;\forall q(x)=\sum_{j=0}^{\cardg} \sum_{i=1}^{l_j} p_{ij}(x)^2g_j(x),\;p_{iJ}\in\R[x],\;\deg q\leq 2r\}.\\    
  \mathrm{ub}_{mq}(f,r)&\textstyle=\inf\{L(f)\mid L\colon\R[x]_{2r}\to\R\text{ is linear, $L(1)=1$, and}\\
    &\qquad \textstyle L(q)\geq0\;\forall q(x)=\sum_{J\subseteq \{1,\dots,\cardg\}} \sum_{i=1}^{l_J} p_{iJ}(x)^2\prod_{j\in J}g_j(x),\;p_{iJ}\in\R[x],\;\deg q\leq 2r\}.
\end{align*}
These satisfy $\mathrm{ub}_{ms}(f,r)\geq\mathrm{ub}_{mq}(f,r)\geq \mathrm{ub}_{mp}(f,r)\geq f_{\min}$, and they correspond, respectively to \emph{sums of squares}, \emph{quadratic module}, and \emph{preordering pseudomoment certificates}.  
}\fi

\begin{table}[!ht]
\centering\begin{tabular}{cccc}
 \hline
 \bf domain $S(\mathbf g)$ & \bf error & \bf certificate & \bf ref. \\
 \hline
 Archimedean & $O(1/\log(r)^c)$ & quadratic module & \cite{nie2007complexity} \\
 Archimedean & $O(1/r^c)$ & quadratic module &\cite{baldi2021moment}\\
 general  & $O(1/r^c)$ & preordering & \cite{schweighofer2004complexity} \\
 general & $O(1/r^c)$ & quadratic module \& uniform denominators & \cite{mai2022complexity} \\
 $S^{n-1}$ & $O(1/r^2)$ & quadratic module / preordering & \cite{fang2021sum} \\
 $\{0,1\}^n$ & see \cite{slot2022binary} & quadratic module / preordering & \cite{slot2022binary} \\
 $B^n$ & $O(1/r^2)$ & quadratic module / preordering & \cite{slot2022sum} \\
 $[-1,1]^n$ & $O(1/r^2)$ & preordering & \cite{laurent2021effective} \\
 $\Delta^n$ & $O(1/r)$ & preordering & \cite{kirschner2022convergence} \\
 $\Delta^n$ & $O(1/r^2)$ & preordering & \cite{slot2022sum} \\
 \hline
 \end{tabular}
 \label{tab:lb}
 \caption{Known results on the asymptotic error of Lasserre's hierarchies of lower bounds; based in part on \cite[Table 1]{slot2022sum}. The domain $S(\mathbf g)$ is assumed to be compact in all cases, $\epsilon\in[0,1/2)$, $c>0$, $S^{n-1}$ is the unit sphere, $B^n$ is the unit ball, $\Delta^n$ is the standard simplex. %
 }
\end{table}

\if{Due to the scaling challenges faced by methods based in hierarchies of the types outlined above, great effort has been invested in understanding possible simplifications when the minimization problem of finding $f_{\min}$ enjoys sparsity or symmetry properties. For the latter, we refer the reader to \cite{riener2013exploiting} and the references therein. Most relevant to the contributions of this paper are the sparsity properties originally proposed in \cite{lasserre2006convergent}, and thoroughly reviewed in \cite[Chapter~3]{magron2022sparse}. 
The two most important sparsity properties studied in the literature are known as term and correlative sparsity, and here we will work in the latter context. Correlative sparsity assumes that $f$ can be written as $f=f_1+f_2+\dots+f_\ell$, where each of the polynomials $f_i$ depends only on a small subset of the variables $x_1,\dots,x_n$, and these subsets satisfy the running intersection property (Definition \ref{def:rip}). Term sparsity instead exploits sparsity of the set of multi-indices $I=(i_1,\dots,i_n)\in \N^n$ appearing as we write $f=\sum_{I}c_{I}x^I$, $c_I\in \R$, $x^I=x_1^{i_1}\dots x_n^{i_n}$; see \cite{wang2021tssos}, \cite[Part II]{magron2022sparse}, as well as  \cite{wang2020cs} where both types of sparsities are combined.

In this paper, we adapt the arguments of Laurent--Slot \cite{slot2021near} to obtain a version of the Schm\"udgen Positivstellensatz on the hypercube $[-1,1]^n$ that takes correlative sparsity into account, that is, we assume the possibility of representing  the function $f$ as a sum of functions involving subsets of variables satisfying the running intersection property. To do this, we adapt the proof of \cite{grimm2007sparse} to obtain a representation as a sum that respects the sparsity and also (almost) respects the lower bounds of $f$. We later follow the strategy of Baldi--Mourrain \cite{baldi2021moment} to obtain a sparse version of the Putinar Positivstellensatz on arbitrary Archimedean domains $S(\mathbf g)$. 

We note that adapting the proof by \cite{grimm2007sparse} has also been done in \cite{mai2022sparse} to obtain a correlative sparsity variant of Reznick's Positivstellensatz.

\fi

{
In this work, we treat the case where the problem possesses the so-called correlative sparsity, where each function $g_i$ depends only on a certain subset of variables and the function $f$ decomposes as a sum of functions depending only on these subsets of variables. This structure can be exploited in order to define sparse lower bounds that are cheaper to compute but possibly weaker.}
Nevertheless, these sparse lower bounds allow one to tackle large-scale polynomial optimization problems arising from various applications including roundoff error bounds in computer arithmetic, quantum correlations and robustness certification of deep networks; see the recent survey \cite{magron2022sparse}.
{
In~\cite{lasserre2006convergent} Lasserre proved that these sparse lower bounds converge as the degree of the SOS multipliers tends to infinity provided the variable groups satisfy the so-called running intersection property (RIP). }
A shorter and more direct proof was provided in \cite{grimm2007sparse}, and was adapted in \cite{mai2022sparse} to obtain a sparse variant of Reznick's Positivstellensatz. 
{
In this work, we show polynomial rates of convergence for sparse hierarchies based on both Schm\"udgen and Putinar Positivstellens\"atze. Importantly, we obtain rates that depend only on the size of the largest clique in the sparsity graph rather than the overall ambient dimension. This allows the perhaps surprising conclusion that, asymptotically, the sparse hierarchy is more accurate than the dense hierarchy for a given computation time of an optimization method, provided that the size of the largest clique is no more than the square root of the ambient dimension. \MK{This assumes that the running time of the optimization method is governed by the size of the largest PSD block and the number of such blocks in the semidefinite programming reformulations of the dense and sparse SOS problems which is the case for the interior point method as well as the most commonly used first-order methods.}

To the best of our knowledge, these are the first quantitative results of this kind. Our proof techniques rely on an adaption of~\cite{grimm2007sparse} and utilize heavily the recent results from~\cite{laurent2021effective} and \cite{baldi2021moment}, and can thus be seen as a generalization of these works to the sparse setting. 
}

The results will be detailed below in Section \ref{sec:result} and further discussed in Section \ref{sec:discussion}, after a brief interlude to establish some notations in Section \ref{sec:notations}.  Some machinery will be developed in Sections \ref{sec:jacksonkernel} and \ref{sec:approximation}, regarding variants of the Jackson kernel and some approximation theory, respectively, and the proofs of the main theorems are presented in Section \ref{sec:proofs}.

\subsection{Notations}
\label{sec:notations}
Denote by $\R$ the set of real numbers, by $\N$ the set of positive integers, and by $\Nzero=\{0,1,\dots\}$ the set of nonnegative integers. Denote by $e_1,\dots,e_n$ the vectors of the standard basis of Euclidean space $\R^n$.

For a Lipschitz continuous function $f\colon [-1,1]^n\to\R$, we set
\[\lip f=\max\left(1,\sup_{x,y\in[-1,1]^n}\frac{|f(x)-f(y)|}{\|x-y\|}\right).\]
We take this to be at least 1 to simplify estimates below.

A multi-index $\multiindex=(i_1,\dots,i_n)\in \Nzero^n$ is an $n$-tuple of nonnegative integers $i_k$, and its weight is denoted by
\[|\multiindex|=\sum_{k=1}^n i_k.\]
For a multi-index $I=(i_1,\dots,i_n)\in \Nzero^n$ and $\J\subset \{1,\dots,n\}$, we will write $I\subseteq \J$ to indicate that for all $1\leq k\leq n$ if $i_k>0$ then $k\in \J$. Similarly, given a multi-index $I\in \Nzero^n$ and a subset $\J\subseteq\Nzero$, we let $\inter{I}{\J}$ be the multi-index whose $k$-th entry is either $i_k$ if $k\in \J$ or $0$ if $k\notin \J$.  %
For two multi-indices $\multiindex$ and $\multiindex'$, we will write $\multiindex\leq \multiindex'$ if the entrywise inequalities $i_k\leq i'_k$ hold for all $\MK{1}\leq k\leq n$. We will distinguish two special multi-indices:
\[\ones=(1,1,\dots,1)\qquad\textrm{and}\qquad\twos=(2,2,\dots,2).\]

We will denote $x^I=x_1^{i_1}x_2^{i_2}\dots x_n^{i_n}$. Also, we denote the \emph{Hamming weight} of $\multiindex\in\Nzero^n$ by
\[\hamming(\multiindex)=\#\{k:i_k>0,\;1\leq k\leq n\}.\]
In other words, $\hamming(\multiindex)$ is the number of nonzero entries in $I$.

We will denote the space of polynomials in $n$ variables by $\polys{}$, and within this set we will distinguish the subspace $\polys{d}$ of polynomials of total degree at most $d$. 
We will denote, for a polynomial $p(x)=\sum_{\multiindex}c_\multiindex x^\multiindex$, by $\fulldeg p$ the vector whose $i$-th entry is the degree of $p$ in $x_i$,
\[\fulldeg p=\big(\max_{c_\multiindex \neq 0}i_1,\max_{c_\multiindex \neq 0}i_2,\dots,\max_{c_\multiindex \neq 0}i_n \big).\]
 Observe that  $\deg p\leq \left|\fulldeg p\right|=
 \sum_{k=1}^n
 \max_{c_I
 \neq 0} i_k$. 
 Set also
  \[\indexsetp=\{\multiindex\in\Nzero^n:c_I\neq0\}.\]
Given a subset $\J\subset \{1,\dots,n\}$, we let $\polysvar{\J}{}$ denote the set of polynomials in the variables $\{x_j\}_{j\in \J}$. For a multi-index $\bfr=(r_1,\dots,r_n)\in\Nzero^n$, we let $\polysvar{}{\bfr}$ denote the set of polynomials $p$ such that, if $p(x)=\sum_{I}c_Ix^I$ for some real numbers $c_I\in \R$, then for each $I=(i_1,\dots,i_n)$ with $c_I\neq 0$ we also have $i_k\leq r_k$ for $1\leq k\leq n$. Finally, we let $\polysvar{\J}{\bfr}=\polysvar{\J}{}\cap\polysvar{}{\bfr}$; in other words, $\polysvar{\J}{\bfr}$ is the set of polynomials $p$ with $\fulldeg p\leq \bfr$ in the variables $\{x_j:j\in \J\}\subseteq \{x_1,\dots,x_n\}$.

Given a set $X$, we will write $X^n$ to denote the product
\[X^n=\underbrace{X\times X\times\dots \times X}_n.\]

We will denote by $\|\cdot\|_\infty$ the supremum norm on $[-1,1]^n$.

The notation $\lceil s\rceil$ stands for the least integer $\geq s$.

 \subsection{Results}
\label{sec:result}

 Let $\sos{\J}$ denote the set of polynomials $p$ that are sums of squares of polynomials in $\polysvar{\J}{}$, that is, of the form $p=p_1^2+\dots+p_\ell^2$ for $p_1,\dots,p_\ell\in \polysvar{\J}{}$.
 
 Let $\cardg\in\N$ and let $\mathbf g=\{g_1,\dots,g_{\cardg}\}$ be a collection of polynomials $g_i\in \polys{}$ defining a set
 \[\domain=\{x\in\R^n:g_i(x)\geq 0,\;i=1,\dots,\cardg\}.\]
 For convenience, denote also $g_0=1$.
 To the collection $\mathbf g$, a multi-index $\bfr$, we associate the 
 \emph{(variable- and degree-wise truncated) quadratic module} associated to the collection $\mathbf g$ and a multi-index $\bfr$ be
 \[\quadraticmoduleg{\bfr}{\mathbf g}{J}=\{\sum_{i=0}^{\cardg}\sigma_ig_i:\sigma_i\in\sos{J},\fulldeg(\sigma_ig_i)\leq \bfr\}.\]
 Similarly, we have a \emph{(variable- and degree-wise truncated) preordering}
 \[\preorderingg{\bfr}{\mathbf g}{J}=\quadraticmoduleg{\bfr}{\{g_K:K\subseteq\{1,\dots,\cardg\}\}}{J}=\{\sum_{K\subseteq\{1,\dots,\cardg\}}\sigma_{K}\constraintg_{K}:\sigma_{K}\in\sos{J},\;\fulldeg(\sigma_{K}\constraintg_{K})\leq \bfr%
 \} \]
 where 
 \[g_K=\prod_{i\in K}g_i.\]

\begin{definition}\label{def:rip}
 A collection $\{\J_1,\dots,\J_\ell\}$ of subsets of $\{1,\dots,n\}\supset \J_j$ satisfies the \emph{running intersection property} if for all $1\leq k\leq \ell-1$ we have
 \[\J_{k+1}\cap\bigcup_{j=1}^k\J_j\subset \J_s\quad\textrm{for some $s\leq k$}.\]
 \end{definition}
Denote, for $j=2,3,\dots,\ell$,
  \begin{equation}\label{eq:definterset}
   \interset_j=J_j\cap\bigcup_{k<j}\J_k.
  \end{equation}
  
\subsubsection{Sparse Schm\"udgen-type representation on \texorpdfstring{$[-1,1]^n$}{}}
\label{sec:schmuedgen}

Let $\maxlip:=\sum_{k=1}^\ell\lip p_k$, $\maxdeg:=\max_{\substack{1\leq k\leq \ell\\1\leq m\leq n}}(\fulldeg p_k)_m$,
  and $\maxJ:=\max_{1\leq k\leq \ell}|\J_k|$.

 \begin{theorem}\label{thm:mainabrv}
  Let $n>0$ and $\ell\geq 2$, and let $\bfr_1,\bfr_2,\dots,\bfr_\ell\in \N^n$, $\bfr_j=(r_{j,1},\dots,r_{j,n})$, 
   be nowhere-vanishing multi-indices. Let also $\J_1,\dots,\J_\ell$ be subsets of $\{1,\dots,n\}$ satisfying the running intersection property.
  Let $p=p_1+p_2+\dots+p_\ell$ be a polynomial that is the sum of finitely many polynomials $p_j\in \polysvar{\J_j}{\bfr_j}$. Then if $p\geq \varepsilon$ on $[-1,1]^n$, we have 
  \[p\in \preordering{\bfr_1}{\J_1}+\dots+\preordering{\bfr_\ell}{\J_\ell}\] 
  as long as, for all $1\leq j\leq \ell$ and all $1\leq i\leq n$,

   \[
   r_j^2\geq \frac{2^{\maxJ+3}(\ell+2)n\pi^2\|p\|_\infty%
    }{\varepsilon}\displaystyle\left(\max\left(\maxdeg,4\Cjackson(\ell+2)\frac{\maxJ\,\maxlip}{\varepsilon}\right) +2\right)^{\maxJ+2} \,.
    \]
   For small enough  $0<\varepsilon<4\Cjackson(\ell+2)\maxJ\,\maxlip/M$, this boils down to
  \[  r_j\geq \sqrt{\frac{A\,\|p\|_\infty}{\varepsilon^{\maxJ+3}}}= O(\varepsilon^{-\frac{\maxJ+3}2}),\]
  with
  \[A=n\pi^2(4\Cjackson\maxJ\,\maxlip+2)^{\maxJ+2}(2(\ell+2))^{\maxJ+3}.\]
\end{theorem}
The proof will be presented in Section \ref{sec:proof}.
\if{
 \begin{theorem}\label{thm:main}
  Let $n>0$ and $\ell\geq 2$, and let $\bfr_1,\bfr_2,\dots,\bfr_\ell\in \N^n$, $\bfr_j=(r_{j,1},\dots,r_{j,n})$, 
   be nowhere-vanishing multi-indices. Let also $\J_1,\dots,\J_\ell$ be subsets of $\{1,\dots,n\}$ satisfying the running intersection property.
  Let $p=p_1+p_2+\dots+p_\ell$ be a polynomial that is the sum of finitely many polynomials $p_j\in \polysvar{\J_j}{\bfr_j}$. Then if $p\geq \varepsilon$ on $[-1,1]^n$, we have 
  \[p\in \preordering{\bfr_1}{\J_1}+\dots+\preordering{\bfr_\ell}{\J_\ell}\] 
  as long as, for all $1\leq j\leq \ell$ and all $1\leq i\leq n$,
  \begin{equation}\label{eq:assumptionepsilon}
  (r_{j,i}+2)^2 \geq\frac{ 2^{\frac{|J_j|}2+2}(\ell+2)\|p\|_\infty n\pi^2}{\varepsilon}\left(\max\left[\max_{1\leq m\leq n}(\fulldeg p_j)_m,\max_{j\leq k\leq \ell}\frac{4\Cjackson(\ell+2)|\interset_k|\sum_{t=k}^\ell\lip p_t}{\varepsilon}
   \right]+2\right)^{|J_j|+2} 
  \end{equation}
 and
 \begin{equation}\label{eq:equivcond}  
  {(r_{j,i}+2)^2}\geq 2\pi^2n\max\left[\max_{1\leq m\leq n}(\fulldeg p_j)_m,\max_{j\leq k\leq \ell}\frac{4\Cjackson(\ell+2)|\interset_k|\sum_{t=k}^\ell\lip p_t}{\varepsilon}
   \right]^2.
 \end{equation}
\end{theorem}
The reader will find the proof of the theorem in Section \ref{sec:proof}.

\begin{remark}[Variable-separated bound]
 The bound in the theorem can be replaced by the following finer version:
 \begin{align*}  
  (r_{j,k}+2)^2 &\geq\frac{ 2^{\frac{|J_j|}2+2}(\ell+2)\|p\|_\infty n\pi^2}{\varepsilon}\prod_{1\leq m\leq n}\left(\max\left[(\fulldeg p_j)_m,\max_{\substack{j\leq l\leq \ell\\ m\in \interset_l}}\frac{4\Cjackson(\ell+2)|\interset_l|\sum_{t=l}^\ell\lip p_t}{\varepsilon}
   \right]+2\right)\\
   &\qquad\cdot\max_{l\in \J_j}\left[(\fulldeg p_j)_l,\max_{\substack{j\leq q\leq \ell\\ l\in \interset_q}}\frac{4\Cjackson(\ell+2)|\interset_q|\sum_{t=q}^\ell\lip p_t}{\varepsilon}
   \right]^{2}.
   \end{align*}
\end{remark}

\begin{remark}[Simplified bound]
\MK{The order in which we present the bounds is a bit odd. We first present the second best bound, then the best and then the worst.} \RRZ{Indeed, the idea was to write the options down so that we could talk about the presentation.}
It suffices to take
  \[
   r_j^2\geq \frac{2^{\maxJ+3}(\ell+2)n\pi^2\|p\|_\infty%
    }{\varepsilon}\displaystyle\left(\max\left(\maxdeg,4\Cjackson(\ell+2)\frac{\maxJ\,\maxlip}{\varepsilon}\right) +2\right)^{\maxJ+2} \,.
    \]
   For small enough  $0<\varepsilon<4\Cjackson(\ell+2)\maxJ\,\maxlip/M$, this boils down to
  \[  r_j\geq \sqrt{\frac{A\,\|p\|_\infty}{\varepsilon^{\maxJ+3}}}= O(\varepsilon^{-\frac{\maxJ+3}2}),\]
  with
  \[A=n\pi^2(4\Cjackson\maxJ\,\maxlip+2)^{\maxJ+2}(2(\ell+2))^{\maxJ+3}.\]
 \end{remark}

}\fi

 \paragraph{Discussion.}\label{sec:discussion}

Solving the dense problem considered by \cite{laurent2021effective} using the sum-of-squares hierarchy reduces to a semidefinite program \MK{with the largest PSD block of size $\binom{n+r}{r}$ that typical optimization methods (e.g., interior point or first order) } can solve in an amount of time proportional to a power of 
\[\binom{n+r}{r}\approx \binom{n+C\varepsilon^{-1/2}}{C\varepsilon^{-1/2}}=:B_{\mathrm{dense} }(\varepsilon).\] 
The bounds we find in Theorem \ref{thm:mainabrv} ---in the case in which $J_j$ is the largest of the sets $J_1,\dots,J_\ell$--- give a bound for the complexity of the leading term as (the same power of)
\[\ell\binom{|\J_j|+|{\bfr_j}|}{|{\bfr_j}|}\leq\ell\binom{|\J_j|(1+C'\varepsilon^{-\frac{|\J_j|+3}2})}{|\J_j|C'\varepsilon^{-\frac{|\J_j|+3}2}}=:B_{\mathrm{sparseSchm}}(\varepsilon).\]
The reason we have $|\bfr_j|\leq |J_j|C'\varepsilon^{-|J_j|-3}$ is that $r_{j,i}\leq O(\varepsilon^{-\frac{|J_j|+3}2})$ and there are at most $|J_j|$ values of $i$ with $r_{j,i}\neq 0$.
\begin{proposition}\label{prop:asympt1}
If $n>|J_j|(|J_j|+3)$ for all $j=1,\dots,\ell$, then we have
 \[\lim_{\varepsilon\searrow0}\frac{B_{\mathrm{sparse Schm}}(\varepsilon)}{B_{\mathrm{dense}}(\varepsilon)}=0.\]
\end{proposition}
\MK{Thus, if the size of the largest clique is of the order of square root of the ambient dimension $n$ or smaller, the sparse bound outperforms the best available dense bound if the performance is measured by the amount of time required by an optimization method to find a bound of a given accuracy $\varepsilon$.
}
 
\begin{proof}[Proof of Proposition \ref{prop:asympt1}]
By Lemma \ref{lem:binomasympt} we have, as $\varepsilon\searrow0$,
\begin{equation*}
\frac{B_{\mathrm{sparse Schm}}(\varepsilon)}{B_{\mathrm{dense}}(\varepsilon)}=\frac{\displaystyle\binom{|\J_j|(1+C'\varepsilon^{-\frac{|\J_j|+3}2})}{|\J_j|C'\varepsilon^{-\frac{|\J_j|+3}2}}}{\binom{n+C\varepsilon^{-1/2}}{C\varepsilon^{-1/2}}}=O(\varepsilon^{\frac12(n-|\J_j|(|J_j|+3))}),
\end{equation*}
and this tends to 0 if the sparsity of the polynomial $p$ is such that $n>|\J_j|(|\J_j|+3)$.
\end{proof}

 \subsubsection{Sparse Putinar-type representation on arbitrary domains}\label{sec:putinar}
 
 Let $n>0$, $\cardg>0$, $\ell\geq 2$, $\J_1,\dots,\J_\ell\subset\{1,\dots,n\}$, $\bfr_1,\dots,\bfr_\ell\in \N^n$, and $p=p_1+\dots+p_\ell$ with $p_j\in \polysvar{\J_j}{\bfr_j}$. Assume that the sets $\J_1,\dots,\J_\ell$ satisfy the running intersection property (Definition \ref{def:rip}).
Let $K_1,\dots,K_\ell\subset \{1,\dots,\cardg\}$ and let $\mathbf g=\{g_1,\dots,g_\cardg\}\subset\polys{}$ be a collection of $\cardg$ polynomials such that, if $i\in K_j$ for some $1\leq j\leq \ell$, then $g_i\in \polysvar{J_j}{}$. Let
\[\domain=\{x\in\R^n:\text{$g_i(x)\geq 0$ for all $i=1,\dots,\cardg$}\}.\]
Denote 
\[\mathbf g_{K_j}=\{g_i:i\in K_j\}.\]
 
 Let $\lojC_1,\dots,\lojC_{\ell}\geq 1$ and $\lojL_1,\dots,\lojL_{\ell}\geq 1$ be constants such that
\[\operatorname{dist}(x,S(\mathbf g_{K_j}))^{\lojL_j}\leq- \lojC_j\min \left\{ \{0\}\cup \left\{g_i(x):i\in K_j\right\} \right\} ,\quad x\in[-1,1]^n;\]
this is a version of the \emph{\L{}ojasiewicz inequality}, and its validity (with appropriate constants $\lojC_j,\lojL_j$) for semialgebraic functions is justified in \cite[Thm.~2.3]{baldi2021moment} and the papers cited therein.

 \begin{theorem}\label{thm:putinar}
Assume that
\begin{equation}\label{eq:gsizeassumption}
 \|g_j\|_\infty\leq\frac12,\quad j=1,2,\dots, \cardg.
\end{equation}
Assume that $\domain\subset [-1,1]^n$ and that there exist polynomials $s_{j,i}\in\polysvar{J_j}{\bfr_j}$, $j=1,\dots,\ell$, $i\in \{0\}\cup K_j$, such that the Archimedean conditions 
\begin{equation}\label{eq:archimedeanity}
 1-\sum_{i\in J_j}x_i^2=s_{j,0}(x)^2+\sum_{i\in K_j}s_{j,i}(x)^2g_i(x),\quad j=1,\dots, \ell,
\end{equation}
hold; that is to say, we assume that $1-\sum_{i\in J_j}x_i^2\in\quadraticmoduleg{\bfr_j}{\mathbf g_{K_j}}{J_j}$.

Then there are constants $\Call>0$, depending only on $\mathbf g$, $J_1,\dots,J_\ell$, such that, if $p\geq \varepsilon>0$ on $S(\mathbf{g})$, we have
 \[p\in \quadraticmoduleg{\bfr_1+\twos}{\mathbf g_{K_j}}{J_j}+\dots+\quadraticmoduleg{\bfr_\ell+\twos}{\mathbf g_{K_\ell}}{J_\ell}\]
 as long as, for all $1\leq j\leq \ell$ and $1\leq k\leq n$,
  \begin{gather}\label{eq:putinarcond1}
  (r_{j,k}+2)^2\geq \Call \frac{4(\ell+2)\left(\sum_i\|p_i\|_\infty \right)^{\lojL_j+1}(\deg p_j\sum_{i=1}^\ell\lip p_i)^{(2\lojL_j+|J_j|+2)(1+\frac{8\lojL_j}3)}}{\varepsilon^{1+\lojL_j+\frac{4\lojL_j+1}3(2\lojL_j+|J_j|+2)(1+\frac{8\lojL_j}3)}},\\
  \label{eq:putinarcond2}
  (r_{j,k}+2)^2\geq \Call\left(\frac{\left(\sum_{i=1}^\ell\|p_i\|_\infty\right)^{\frac{4\lojL_j+1}3}(\deg p_j\sum_{i=1}^\ell\lip p_i)^{\frac{8\lojL_j}3}}{ \varepsilon^{\frac{12\lojL_j+1}3}}\right)^2
 \end{gather}
\end{theorem}
The proof of the theorem can be found in Section \ref{sec:putinarproof}.

\paragraph{Discussion.} By the same arguments we used in the discussion at the end of the previous section, if we assume $\lojL_1=\dots=\lojL_{\cardg}=1$,  the bounds we find in Theorem \ref{thm:putinar} give a bound for the complexity of the leading term as (a power of) 
\[\ell\binom{|\J_j|+|{\bfr_j}|}{|{\bfr_j}|}\approx \binom{|J_j|(1+C''\varepsilon^{-\frac{26}3-\frac53|J_j|})}{|J_j|C''\varepsilon^{-\frac{26}3-\frac53|J_j|}})=:B_{\textrm{sparsePut}}.\]
The assumption $\lojL_1=\dots=\lojL_{\cardg}=1$ is realized  for example when the so-called \emph{constraint qualification condition} that, at each point $x\in S(\mathbf g)$ all the active constraints $g_{i_1},\dots,g_{i_l}$ (i.e., those satisfying $g_{i_j}(x)=0)$ have linearly independent gradients $\nabla g_{i_1}(x),\dots,\nabla g_{i_l}(x)$), holds; this latter statement is proved in  \cite[Thm~2.11]{baldi2021moment}.

In this case, we have:
\begin{proposition}\label{prop:asympt2}
  If $\frac n2>|J_j|(\frac{26}3+\frac53|J_j|)$ for all $j=1,\dots,\ell$ and if $\lojL_1=\dots=\lojL_{\cardg}=1$, then we have
  \[\lim_{\varepsilon\searrow0}\frac{B_{\mathrm{sparsePut}}}{B_{\mathrm{dense}}}=0.\]
\end{proposition}

\MK{Again the implication is that the sparse bound asymptotically outperforms the dense bound  provided that the largest clique is sufficiently small}.
\begin{proof}[Proof of Proposition \ref{prop:asympt2}]
Lemma \ref{lem:binomasympt} gives
\begin{equation*}
\frac{B_{\mathrm{sparsePut}}}{B_{\mathrm{dense}}}=
\frac{\displaystyle\binom{|\J_j|(1+C'\varepsilon^{-\frac{26}3-\frac53|J_j|})}{|\J_j|C'\varepsilon^{-\frac{26}3-\frac53|J_j|}}}{\binom{n+C\varepsilon^{-1/2}}{C\varepsilon^{-1/2}}}=O(\varepsilon^{\frac n2-|J_j|(\frac{26}3+\frac53|J_j|)}),
\end{equation*}
which tends to 0 if $\frac n2>|J_j|(\frac{26}3+\frac53|J_j|)$. 
\end{proof}

\paragraph{Organization of the paper.} 
The proof of Theorem \ref{thm:mainabrv} can be seen as a variable-separated version of the proof in \cite{laurent2021effective}, which relies on the Jackson kernel. Therefore in Section \ref{sec:jacksonkernel} we derive the suitable ingredients for sparse Jackson kernels while carefully taking into account each variable separately. 

A strategy is also required to write a positive polynomial $p$ that is known to be a sum $p=p_1+\dots+p_\ell$ with $p_i\in\polysvar{J_i}{}$ as a similar sum $p=h_1+\dots+h_\ell$ but now with $h_j\in \polysvar{J_j}{}$ and $h_j\geq 0 $ on $[-1,1]^{|J_j|}$; this is done in Section \ref{sec:approximation}.

Section \ref{sec:proofs} gives the proofs of Theorems \ref{thm:mainabrv} and \ref{thm:putinar}, together with the statement and proof of Lemma \ref{lem:binomasympt}, which was used in the proofs of Propositions \ref{prop:asympt1} and \ref{prop:asympt2} above.

\section{The sparse Jackson kernel}
\label{sec:jacksonkernel}

The measure $\chebyshevmeas{n}$ on the box $[-1,1]^n$ defined by
\[ d\chebyshevmeas{n}(x)\coloneqq\frac{dx_1}{\pi\sqrt{1-x_1^2}}\cdots \frac{dx_n}{\pi\sqrt{1-x_n^2}},\quad x=(x_1,\dots,x_n)\in[-1,1]^n,\]
is known as the \emph{(normalized) Chebyshev measure}; it is a probability measure on $[-1,1]^n$. It induces the inner product
\[\inner{f}{g}{\chebyshevmeas n}\coloneqq\int_{[-1,1]^n}f(x)g(x)\,d\chebyshevmeas n(x)\]
and the norm $\|f\|_{\chebyshevmeas{n}}=\sqrt{\inner{f}{f}{\chebyshevmeas{n}}}$.

For $k=0,1,\dots$, we let $T_k\in \polys{}$ be the univariate \emph{Chebyshev polynomial} of degree $k$, defined by
\[T_k(\cos\theta)\coloneqq\cos(k\theta),\quad\theta\in\R.\]
The Chebyshev polynomials satisfy $|\chebyshevpoly_k(x)|\leq 1$ for all $x\in[-1,1]$, and 
\[\inner{T_a}{T_b}{\chebyshevmeas 1}=\int_{-1}^1\frac{\chebyshevpoly_a(x)\chebyshevpoly_b(x)}{\pi\sqrt{1-x^2}}dx=\begin{cases}0,&a\neq b,\\1,&a=b=0,\\\frac12,&a=b\neq 0.\end{cases}\]
For a multi-index $\multiindex=(i_1,\dots,i_n)$, we let
\[\chebyshevpoly_\multiindex(x_1,\dots,x_n)\coloneqq \chebyshevpoly_{i_1}(x_1)T_{i_2}(x_2)\cdots \chebyshevpoly_{i_n}(x_n)\]
be the multivariate Chebyshev polynomials, which then satisfy (see for example \cite[\S II.A.1]{weisse2006kernel}), for multi-indices $\multiindex$ and $\multiindex'$,
\begin{equation}\label{eq:inner}
 \deg\chebyshevpoly_\multiindex=|\multiindex|\quad\textrm{and}\quad \inner{\chebyshevpoly_\multiindex}{\chebyshevpoly_{\multiindex'}}{\chebyshevmeas n}=\begin{cases}
                                    0,&\multiindex\neq \multiindex',\\
                                    2^{-\hamming(\multiindex)},&\multiindex=\multiindex'.
                                   \end{cases}
\end{equation}                    
Thus $p\in\polys{d}$ can be expanded as $p=\sum_{|\multiindex|\leq d}2^{\hamming(\multiindex)}\inner{p}{\chebyshevpoly_\multiindex}{\chebyshevmeas n}\chebyshevpoly_\multiindex$.

If we let, %
for a finite collection $\Lambda\subseteq\R\times\Nzero^n$ of pairs $(\lambda,\multiindex)$ of a real number $\lambda$ and a multi-index~$\multiindex$,
\[\kernel{}{\Lambda}(x,y)=\sum_{(\lambda,I)\in \Lambda}2^{w(\multiindex)}\lambda\, \chebyshevpoly_\multiindex(x)\chebyshevpoly_\multiindex(y),\quad x,y\in\R^n,\]
then, for any $p\in\polys{}$, we have
\[\kernelop{}{\Lambda}(p)(x)\coloneqq\int_{[-1,1]^n}\kernel{}{\Lambda}(x,y)p(y)\,d\chebyshevmeas n(y)=\sum_{\substack{(\lambda,I)\in\Lambda\\|I|\leq d}} 2^{\hamming(\multiindex)}\lambda \inner{p}{\chebyshevpoly_\multiindex}{\chebyshevmeas{n}} \chebyshevpoly_\multiindex(x).\]
This means that, if we set all the nonzero numbers $\lambda$ equal to 1, then $\kernelop{}{\Lambda}$ is the identity operator in the linear span of $ \{T_\multiindex:{\exists \lambda\ne 0\;\mathrm{s.t.}\;}(\lambda,I)\in \Lambda\}\subseteq\polys{d}$. 

We let, for $r,k\in\N$,
\[\lambda^r_k=\frac{1}{r+2}\left((r+2-k)\cos\tfrac{\pi k}{r+2}+\frac{\sin\frac{\pi k}{r+2}}{\sin\frac{\pi }{r+2}}\cos\tfrac{\pi }{r+2}\right),\quad 1\leq k\leq r,\]
and 
\[\lambda^r_0=1.\]
We set, for $\bfr=(r_1,\dots,r_n)\in\Nzero^n$,
 \[\lambda^\bfr_\multiindex=\prod_{j=1}^n\lambda_{i_j}^{r_j}\]
and
\[\Lambda_\bfr=\{(\lambda^\bfr_\multiindex,\multiindex):I\leq\bfr\}.\]
Then  $\jkernel{\bfr}=\kernel{}{\Lambda_\bfr}$ is the \emph{($\bfr$-adapted) Jackson kernel}, and its associated linear operator $\kernelop{}{\Lambda_\bfr}$ will be denoted $\jkernelop{\bfr}$.
\begin{theorem}\label{thm:propsK}
 We have $\jkernelop{\bfr}(\polysvar{\J}{\bfr})\subseteq \polysvar{\J}{\bfr}$, and if $p(x)\geq 0$ on $[-1,1]^n$ then  $\jkernelop{\bfr}(p)\geq 0$ on $[-1,1]^n$.
 
 Also, we have:
 \begin{enumerate}[label=P\arabic*.,ref=P\arabic*]
  \item \label{P1} If $p\in\polysvar{\J}{\bfr}$ satisfies $p(x)\geq 0$ for all $x\in[-1,1]^n$,
  \[\jkernelop{\bfr}(p)\in\preordering{\bfr}{\J}.\]
  \item \label{P2} Let $p\in \polysvar{\J}{\bfr}$ be a polynomial that satisfies $0\leq p(x)\leq 1$ for all $x\in[-1,1]^n$, and  %
  for all $I=(i_1,\dots,i_n)\in \indexset{p}$, assume that $\bfr=(r_1,\dots,r_n)$ verifies
  \begin{equation}\label{eq:effcond}
   \frac{i_j^2}{(r_j+2)^2}\leq \frac{1}{2\pi^2n},\quad 1\leq j\leq n.
  \end{equation}
  Assume that
  \[\varepsilon\geq2n\pi^2\left(\prod_{1\leq k\leq n}((\fulldeg p)_k+1)\right)\max_{\multiindex\in\indexsetp}\left[2^{\hamming(\multiindex)/2}\max_j\frac{i_j^2}{(r_j+2)^2}\right]>0.\]
  Then 
  \[\|(\jkernelop{\bfr})^{-1}(p+\varepsilon)-(p+\varepsilon)\|_\infty\leq \varepsilon.\]
 \end{enumerate}

\end{theorem}
\begin{proof}
 The first statement of the theorem corresponds to Lemma \ref{lem:propsK}\ref{it:2} and \ref{lem:propsK}\ref{it:9}.
 Property \ref{P2} follows from Lemma \ref{lem:propsK}\ref{it:10}.
 
 Let us prove property \ref{P1}. Take a finite subset $\{z_i\}_{i}$ of $[-1,1]^n$ and a corresponding set of positive weights $\{w_i\}_i\subset\R$ giving a quadrature rule for integration of polynomials $q\in \polysvar{\J}{\bfr}$, so that
 \[\int_{[-1,1]^n}q(x)\,d\chebyshevmeas{n}(x)=\sum_iw_iq(z_i),\quad q\in\polysvar{\J}{\bfr}.\]
 Then we have, for $p$ as in the statement of \ref{P1}, 
 \[\jkernelop{\bfr}(p)(x)=\sum_{i}w_ip(z_i)\jkernel{\bfr}(z_i,x),\]
 with $w_ip(z_i)\geq 0$. Since, by  Lemma \ref{lem:propsK}\ref{it:5} and Theorem \ref{thm:karlinshapley} below, $\jkernel{\bfr}(z_i,x)$ is in $\preordering{\bfr}{\J}$, so is $\jkernelop{\bfr}(p)$.
\end{proof}

\begin{corollary}\label{coro:puttingtogether}
 If $p\in\polysvar{\J}{\bfr}$ satisfies $0\leq p(x)\leq 1$ for all $x\in [-1,1]^n$, %
 then 
 \[p+\varepsilon\in\preordering{\bfr}{\J}\]
 for all multi-indices $\bfr$ satisfying \eqref{eq:effcond} and
 \[\varepsilon\geq 2n\pi^2\left(\prod_{1\leq k\leq n}((\fulldeg p)_k+1)\right)\max_{I\in\indexsetp}\left(2^{\hamming(I)/2}\max_{1\leq j\leq n}\frac{i_j^2}{(r_j+2)^2}\right).\]
 Here, $\bfr=(r_1,\dots,r_n)$ and $\multiindex=(i_1,\dots,i_n)$.
\end{corollary}
\begin{proof}
 By property \ref{P2} in Theorem \ref{thm:propsK}, 
 \[\|(\jkernelop{\bfr})^{-1}(p+\varepsilon)-(p+\varepsilon)\|_\infty\leq \varepsilon.\]
 Thus, $(\jkernelop{\bfr})^{-1}(p+\varepsilon)\geq0$ on $[-1,1]^n$.
 By property \ref{P1} and Lemma \ref{lem:propsK}\ref{it:2},
 \[p+\varepsilon=\jkernelop{\bfr}\circ(\jkernelop{\bfr})^{-1}(p+\varepsilon)\in\preordering{\bfr}{\J}.\qedhere\]
\end{proof}

The rest of this section is devoted to results used in the proof of Theorem \ref{thm:propsK}.

\begin{theorem}[{\cite[Th.~10.3]{karlin1953geometry}}]\label{thm:karlinshapley}
If $p \in \R[y]$ is a univariate polynomial of degree $d$ nonnegative on the interval $[a,b] \subset \R$, then
\[
\begin{cases}
p = \sigma_0 + \sigma_1(b-y)(y-a),\quad \sigma_0\in \Sigma_d[y], \;\;\;\;\,\sigma_1 \in \Sigma_{d-2}[y] & d \mathrm{\;\, even,} \\
p = \sigma_0(y-a) + \sigma_1(b-y),\quad \sigma_0\in \Sigma_{d-1}[y], \;\sigma_1 \in \Sigma_{d-1}[y] & d \mathrm{\;\, odd},
\end{cases}
\]
where $\Sigma_d$ is the cone of sum-of-squareds of polynomials of degree at most $d$.
\end{theorem}

\begin{lemma}\label{lem:propsK}
 Let $\bfr\in\Nzero^n$ be a multi-index.
 The operator $\jkernelop{\bfr}$ defined above has the following properties:
 \begin{enumerate}[label=\roman*.,ref=(\textit{\roman*})]
  \item \label{it:2} $\jkernelop{\bfr}(\polysvar{\J}{\bfr})\subseteq \polysvar{\J}{\bfr}$.
  \item \label{it:3} %
  We have
  \[\jkernelop{\bfr}(\chebyshevpoly_\multiindex)=\begin{cases}
                                                    \lambda^\bfr_\multiindex \chebyshevpoly_\multiindex, & \multiindex\leq \bfr,\\
                                                    0,&\textrm{otherwise.}
                                                   \end{cases}\]
            In particular, $\jkernelop{\bfr}(1)=1$.
  \item \label{it:4} $\jkernelop{\bfr}$ is invertible in $\polysvar{\J}{\bfr}$ with $\J=\{i:1\leq i\leq n,\;r_i>0\}$.
  \item \label{it:6} $0< \lambda^\bfr_\multiindex\leq 1$ for all $0\leq \multiindex\leq \bfr$.
  \item \label{it:7}  For $\multiindex=(i_1,\dots,i_n)$ and $ \bfr=(r_1,\dots,r_n)$ in $\Nzero^n$,
  \[|1-\lambda^\bfr_\multiindex|=1-\lambda^\bfr_\multiindex\leq n\pi^2\max_j\frac{i_j^2}{(r_j+2)^2}.\]
  \item \label{it:8} For $\multiindex=(i_1,\dots,i_n)$ and $ \bfr=(r_1,\dots,r_n)$ in $\Nzero^n$ 
  that verify \eqref{eq:effcond}, %
  we have
  \[\left|1-\frac1{\lambda^\bfr_\multiindex}\right|\leq 2n\pi^2\max_j\frac{i_j^2}{(r_j+2)^2}.\]%
  \item \label{it:10} Let $p%
  \in \polys{\J,\bfr}$ with $p(x)\geq 0$ for all $x\in[-1,1]^n$ and $\|p\|_\infty\leq1$.
  Let $\bfr=(r_1,\dots,r_n)$ be a multi-index such that $I\leq \bfr$ for all $I\in\indexset{p}$, and assume that, for all $I=(i_1,\dots,i_n)\in \indexset{p}$, 
  condition \eqref{eq:effcond} is verified.
  Then we have
  \[\|(\jkernelop{\bfr})^{-1}(p)-p\|_\infty \leq 2n\pi^2\left(\prod_{1\leq k\leq n}((\fulldeg p)_k+1)\right)\max_{\substack{\multiindex\in\indexset{p}\\1\leq j\leq n}}\left[2^{\hamming(\multiindex)/2}\frac{i_j^2}{(r_j+2)^2}\right].\]
  \item \label{it:5} $\jkernel{\bfr}(x,y)\geq 0$ for all $x,y\in[-1,1]^n$.
  \item \label{it:9} If $p\in\polys{}$ is such that $p(x)\geq 0$ for $x\in[-1,1]^n$, then $\jkernelop{\bfr}(p)(x)\geq 0$ for all $x\in[-1,1]^n$.
 \end{enumerate}
\end{lemma}
\begin{proof} Throughout, we follow \cite{laurent2021effective}.

  Item \ref{it:3} is immediate from the definitions and \eqref{eq:inner}.
  Item %
  \ref{it:2} follows from item \ref{it:3} and the fact that $\{\chebyshevpoly_\multiindex:\multiindex\leq \bfr,\;\multiindex\subseteq \J\}$ is a basis for $\polysvar{\J}{\bfr}$. 
  
  Observe that item \ref{it:3} means that $\jkernelop{\bfr}$ is diagonal in $\polysvar{\J}{\bfr}$, so in order to prove item \ref{it:4} it suffices to show that $\lambda^\bfr_\multiindex>0$ for all $I\leq \bfr$, $I\subseteq \J$. This follows immediately from item \ref{it:6}, which in turn follows from the definition of $\lambda^\bfr_\multiindex$ and \cite[Proposition 6(ii)]{laurent2021effective}, which shows that $0<\lambda_k^r\leq 1$ for all $0\leq k\leq r$.

  Similarly, by \cite[Proposition 6(iii)]{laurent2021effective} we have that, if $k\leq r$, then
  \[|1-\lambda^r_k|=1-\lambda^r_k\leq \frac{\pi^2 k^2}{(r+2)^2}.\]
  Thus, if $\gamma_j=1-\lambda_{i_j}^{r_j}\leq \pi^2i_j^2/(r_j+2)^2 $ and $\gamma=\max_j\gamma_j$, we also have, using Bernoulli's inequality \cite[Lemma 11]{laurent2021effective}
  \begin{align*}
   1-\lambda^\bfr_\multiindex&=1-\prod_{j=1}^n\lambda^{r_j}_{i_j}\\
   &=1-\prod_{j=1}^n(1-\gamma_j)\\
   &\leq 1-(1-\gamma)^n \\
   &\leq n\gamma\\
   &\leq n\pi^2\max_j\frac{i_j^2}{(r_j+2)^2}
  \end{align*}  
  This shows item \ref{it:7}. Using it, we can prove item \ref{it:8} as follows: condition \eqref{eq:effcond} %
  implies, by item \ref{it:7}, that $|1-\lambda^\bfr_\multiindex|\leq 1/2$, and hence $|\lambda^\bfr_\multiindex|\geq 1/2$, so
  \[\left|1-\frac1{\lambda^\bfr_\multiindex}\right|=\frac{|1-\lambda^\bfr_\multiindex|}{|\lambda^\bfr_\multiindex|}\leq  2n\pi^2\max_j\frac{i_j^2}{(r_j+2)^2},\]%
  leveraging item \ref{it:7} again.

   Let us show item \ref{it:10}.
  From items \ref{it:3} and \ref{it:4}, we have
  \begin{align*}
   \|(\jkernelop{\bfr})^{-1}(p)-p\|_\infty
   &=\left\|\sum_{\multiindex}\left[\frac1{\lambda^\bfr_\multiindex}2^{\hamming(\multiindex)}\inner{p}{\chebyshevpoly_\multiindex}{\chebyshevmeas{n}}\chebyshevpoly_{\multiindex}-2^{\hamming(\multiindex)}\inner{p}{\chebyshevpoly_\multiindex}{\chebyshevmeas{n}}\chebyshevpoly_{\multiindex}\right]\right\|_\infty\\
   & \leq  \sum_{\multiindex}2^{\hamming(\multiindex)}|\inner{p}{\chebyshevpoly_\multiindex}{\chebyshevmeas{n}}|\left|1-\frac1{\lambda^\bfr_\multiindex}\right|,
  \end{align*}
  because $|\chebyshevpoly_\multiindex(x)|\leq 1$ for all $x\in [-1,1]^n$.
  Plugging in the estimate from item \ref{it:8}, we get
  \begin{align*} 
   \|(\jkernelop{\bfr})^{-1}(p)-p\|_\infty&\leq \sum_{\multiindex}2^{\hamming(\multiindex)/2+1}n\pi^2\max_j\frac{i_j^2}{(r_j+2)^2}\\
   & \leq 2n\pi^2\left(\prod_{1\leq k\leq n}((\fulldeg p)_k+1)\right)\max_{\multiindex\in\indexset{p}}\left[2^{\hamming(\multiindex)/2}\max_j\frac{i_j^2}{(r_j+2)^2}\right]
  \end{align*}
  where we have also used 
  \begin{align*}
   |\inner{p}{\chebyshevpoly_\multiindex}{\chebyshevmeas{n}}|\leq \|p\|_{\chebyshevmeas{n}}\|\chebyshevpoly_\multiindex\|_{\chebyshevmeas{n}}\leq \|\chebyshevpoly_\multiindex\|_{\chebyshevmeas{n}}=2^{-\hamming(\multiindex)/2},
  \end{align*}
  which follows from \eqref{eq:inner}. 
  
  To prove item \ref{it:5}, let, for fixed $\bfr$, 
  \begin{align*}
   \Lambda_k&=\{(\lambda_\multiindex^\bfr,\multiindex):\multiindex\leq (0,\dots,0,r_k,0,\dots,0)\}\\
   &=\{(\lambda_{i_k}^{r_k},(0,\dots,0,i_k,0,\dots,0)):i_k\leq r_k\},\quad 1\leq k\leq n,
  \end{align*}
  and observe that
  \begin{equation}\label{eq:composition}
   \jkernelop{\bfr}=\kernelop{}{\Lambda_\bfr}=\kernelop{x_1}{\Lambda_{1}}\circ \kernelop{x_2}{\Lambda_{2}}\circ \dots\circ \kernelop{x_n}{\Lambda_{n}}
  \end{equation}
  where $\kernelop{x_k}{\Lambda_{k}}$ is the operator $\kernelop{}{\Lambda_{k}}$ acting in the variable $x_k$, i.e.,
  \[\kernelop{x_k}{\Lambda_k}(p)(x)=\int_{-1}^{1}\kernel{}{\Lambda_k}(x_k,y)\,p(x_1,\dots,x_{k-1},y,x_{k+1},\dots,x_n)\,d\chebyshevmeas{1}(y).\]
  Equation \eqref{eq:composition} follows from the identity
  \begin{align*}
   \kernel{}{\Lambda_\bfr}(x,y)&=\kernel{}{\Lambda_{1}}(x_1,y_1)\kernel{}{\Lambda_{2}}(x_2,y_2)\cdots\kernel{}{\Lambda_{n}}(x_n,y_n)\\
   &=\jkernel{{(r_1)}}(x_1,y_1)\jkernel{{(r_2)}}(x_2,y_2)\cdots\jkernel{{(r_n)}}(x_n,y_n)
  \end{align*}
  that can be checked from the definitions.   
  Item \ref{it:5} then follows from the well-known fact that $\jkernel{(r)}(x,y)\geq 0$ for all $r\in \Nzero$ and all $x,y\in[-1,1]$; see for example \cite[\S II.C.2--3]{weisse2006kernel}. 
  
  Item \ref{it:9} follows immediately from item \ref{it:5}.
\end{proof}

\section{Sparse approximation theory}
\label{sec:approximation}

For $1\leq i\leq n$ and a function $f\colon[-1,1]^n\to\R$, let
\[\lip_if=\sup_{\substack{x\in[-1,1]^n\\y\in[-1,1]}}\frac{|f(x)-f(x_1,\dots,x_{i-1},y,x_{i+1},\dots,x_n)|}{|x_i-y|}.\]

\begin{theorem}\label{thm:jackson}
 There is a constant $\Cjackson >0$ such that the following is true.
 Let $f\in C^0([-1,1]^n)$ be a Lipschitz function with variable-wise Lipschitz constants $\lip_1 f,\dots,\lip_nf$. Then for each multi-index $\mathbf m=(m_1,\dots,m_n)\in\N^n$ there is a polynomial $p\in \polys{\mathbf m}$ such that
 \[\sup_{x\in[-1,1]^n}|f(x)-p(x)|\leq \Cjackson\sum_{i=1}^n\frac{\lip_i f}{m_i}\]
 and 
 \[\lip_i p\leq 2\lip_i f.\]
\end{theorem}
\begin{proof}
 Jackson \cite[p.~2--6]{jackson1930theory} proved that there is a constant $C>0$ such that, if $g\colon\R\to\R$ is Lipschitz and $\pi$-periodic, $g(0)=g(\pi)$, then
 \begin{align}
 \label{eq:jacksonone}
 \left|g(\theta)-\frac{\int_{-\pi/2}^{\pi/2} g(\theta-\vartheta)\left(\frac{\sin m\vartheta}{m\sin \vartheta}\right)^4d\vartheta}{\int_{-\pi/2}^{\pi/2}\left(\frac{\sin m\vartheta}{m\sin \vartheta}\right)^4d\vartheta}\right|\leq \frac{C\lip g}{m},\quad m\in \N, x\in\R.
 \end{align}
 For a multivariate Lipschitz function $g\colon \R^n\to\R$ and a multi-index $\mathbf m=(m_1,\dots,m_n)\in\N^n$, let
 \[L_i(g)(\theta)=\frac{\int_{-\pi/2}^{\pi/2}\dots\int_{-\pi/2}^{\pi/2} g(\theta_1-\vartheta_1,\dots,\theta_i-\vartheta_i,\theta_{i+1},\dots,\theta_n)\prod_{j=1}^i\left(\frac{\sin m_j\vartheta_j}{m_j\sin \vartheta_j}\right)^4d\vartheta_1\dots d\vartheta_i}{\prod_{j=1}^i\int_{-\pi/2}^{\pi/2}\left(\frac{\sin m_j\vartheta_j}{m_j\sin \vartheta_j}\right)^4 d\vartheta_j}.\]
 Then we have, using the triangle inequality and the single-variable inequality \eqref{eq:jacksonone} at each step,
 \begin{multline*}
  |g(\theta)-L_n(g)(\theta)|\\
  \leq |g(\theta)-L_1(g)(\theta)|+|L_1(g)(\theta)-L_2(g)(\theta)|+\dots+|L_{n-1}(g)(\theta)-L_{n}(g)(\theta)|\\
  \leq C\left(\frac{\lip_1g}{m_1}+\dots+\frac{\lip_n g}{m_n}\right).
 \end{multline*}
 The function $\prod_j(\sin m_j\theta_j/m_j\sin \theta_j)^4$ is a polynomial of degree $m_j$ in $\cos \theta_j$ (cf. \cite[p.~3]{jackson1930theory}). If we replace $f$ with its Lipschitz extension to $[-2,2]^n$ and apply the results above to $g(\theta)=f(2\cos \theta_1,\dots,2\cos \theta_n)$ we get a polynomial $L_n(g)(\theta)$ in $\cos \theta_1,\dots,\cos \theta_n$ satisfying the above inequality. Thus
 \[p(x)=L_n(g)(\arccos(x_1/2),\dots,\arccos (x_n/2)),\quad x\in [-2,2],\] 
 is a polynomial with $\fulldeg p\leq\mathbf m$ that satisfies (cf. \cite[p.~13--14]{jackson1930theory})
 \[|f(x)-p(x)|\leq C\left(\frac{\lip_1 g}{m_1}+\dots+\frac{\lip_n g}{m_n}\right)\leq2C\left(\frac{\lip_1 f}{m_1}+\dots+\frac{\lip_n f}{m_n}\right),\]
 since $\lip_i g\leq2\lip_i f$. This proves the first statement, setting $\Cjackson=2C$.
 We also have 
 \[\left|\frac{d}{dx}\arccos(x/2)\right|=\frac{1}{2\sqrt{1-(x/2)^2}}\leq \frac1{\sqrt3}\quad\text{for}\quad x\in[-1,1]\] 
 and, by linearity and monotonicity of $L_n$,
 \begin{multline*}|L_ng(\theta)-L_ng(\theta_1,\dots,\theta_{i-1},\theta_i+t,\theta_{i+1},\dots,\theta_n)|\leq 
 \left|L_n(|t|\lip_i g)(\theta)\right|=
|t|\lip_i g\leq 2|t|\lip_i f,
 \end{multline*}
 whence 
 \begin{multline*}
  \lip_ip=\lip_iL_n(g)(\arccos(x_1/2),\dots,\arccos(x_n/2))\\
  \leq \lip_i L_n(g)\left|\frac{d}{dx_i}\arccos\frac {x_i}2\right|\leq \lip_iL_n(g)\leq 2\lip_if.
 \end{multline*}
 \qedhere
\end{proof}

\begin{lemma}[{a version of \cite[Lemma 3]{grimm2007sparse}}]\label{lem:approx}
 Let $\J_1,\dots, \J_\ell$ be subsets of $\{1,\dots,n\}$ satisfying the running intersection property.
 Suppose $f=f_1+\dots+f_\ell$ with $\ell\geq 2$, $f_j\in \polysvar{\J_j}{}$. %
 Let $\varepsilon>0$ be such that $f\geq \varepsilon$ on $\domain\subseteq [-1,1]^n$. 
 Pick numbers $\epsilon,\eta>0$ so that
 \[
                   \varepsilon=(\ell-1)\epsilon-(\ell-2)\eta\quad\textrm{and}\quad \epsilon>2\eta.
               \]
 Set, for $2\leq l\leq \ell$,
 \begin{equation}\label{eq:defD}
  D_{l,m}=%
  \left\lceil\frac{2\,\Cjackson\,|\interset_l|\sum_{k=l}^m\lip f_k}{\epsilon-2\eta}\right\rceil
 \end{equation}
 with $\interset_l$ as in \eqref{eq:definterset}, and $D_{1,m}=D_{2,m}$.

 Then $f=h_1+\dots+h_\ell$ for some $h_j\in   \polysvar{\J_j}{}$ with $h_j\geq \eta$ on $\domain\subseteq[-1,1]^n$ and

\begin{equation}\label{eq:entrywise}
  \fulldeg h_j\leq \inter{\max(\fulldeg f_j,\bar D_{j,\ell},\bar D_{j+1,\ell},\dots, \bar D_{\ell,\ell})}{\J_j}
 \end{equation}
 where $\bar D_{j,m}$ is the multi-index whose $k$-th entry equals $D_{j,m}$ if $k\in \interset_j=\J_j\cap \bigcup_{k<j}\J_i$ %
 and 0 otherwise, and the maximum is taken entry-wise.
 
 Additionally, if $\lip f$ denotes the Lipschitz constant of $f$ on $[-1,1]^n$, then  %
 \[\lip h_j\leq 3\sum_{k=j}^\ell\lip f_k.\]
 Finally, we have 
 \[\|h_j\|_\infty\leq 3\times 2^{\ell-1}\sum_{j=1}^\ell\|f_j\|_\infty.\]
\end{lemma}
\begin{remark}
    If $\domain=[-1,1]^n$, we also have the obvious estimate $\|h_j\|_\infty\leq \|f\|_\infty$, that follows from $0\leq h_j\leq f$.
\end{remark}
\begin{proof}
 In order to prove the result by induction, let us first consider the case $\ell=2$. In this case, $\varepsilon=\epsilon$ and $\epsilon>2\eta$. Assume that $\J_1\cap \J_2\neq \emptyset$. For a subset $J\subset\{1,\dots,n\}$, let $\pi_J$ denote the projection onto the variables with indices in $J$, that is, $\pi_J(x)=(x_i)_{i\in J}\in[-1,1]^J$ for $x\in [-1,1]^n$.
 Define $g\colon[-1,1]^{\J_1\cap\J_2}\to\R$ by 
 \[g(x):=\min_{y\in\pi_{\J_1\setminus\J_2}(\domain)\subseteq [-1,1]^{\J_1\setminus\J_2}} f_2(x,y)-\frac\varepsilon2, \qquad x\in [-1,1]^{\J_1\cap \J_2}.\]
 
 The function $g$ is Lipschitz continuous on $[-1,1]^{\J_1\cap\J_2}$. To see why, let $x,x'\in [-1,1]^{\J_1\cap\J_2}$ and pick  $y,y'\in\pi_{\J_1\setminus\J_2}(\domain)\subseteq [-1,1]^{\J_1\setminus\J_2}$ minimizing $f_2(x,y)$ and $f_2(x',y')$, respectively. Then 
 \begin{multline*}
  |g(x)-g(x')|=|f_2(x,y)-f_2(x',y')|\\
  \leq \max(|f_2(x,y)-f_2(x',y)|,|f_2(x,y')-f_2(x',y')|)\leq \lip (f_2)|x-x'|,
 \end{multline*}
 where $\lip (f_2)$ denotes the Lipschitz constant of $f_2$ on $[-1,1]^{n}$. 
 
 The function $g$ also satisfies
 \[ f_1+g\geq \frac\varepsilon2 \quad\textrm{and}\quad f_2-g\geq \frac\varepsilon2
 \]
 on $\domain$. The second inequality follows from the definition of $g$, and the first one can be shown taking 
 $(x,y,z)\in \domain$ with  $x\in [-1,1]^{\J_1\cap\J_2}$,  $y\in [-1,1]^{\J_1\setminus \J_2}$, and $z\in[-1,1]^{\J_2\setminus\J_1}$, taking care to pick $y$ only after $x$ has been chosen, in such a way that the minimum is in the definition of $g$ is realized there, that is, $g(x)=f_2(x,y)-\varepsilon/2$ holds (this is possible by compactness of $\domain$ and continuity of $f$); then we have
 \[f_1(x,z)+g(x)=f_1(x,z)+f_2(x,y)-\frac\varepsilon2=f(x,y,z)-\frac\varepsilon2\geq \frac\varepsilon2.\]
  
 For $j\in \J_1\cap \J_2$, let %
 \[m_j= D_{2,2}=\left\lceil\frac{2\Cjackson\,|\J_1\cap\J_2|\lip (f_2)}{\varepsilon-2\eta}\right\rceil.\]
 Set $m_j=0$ for all other $0\leq j\leq n$, and $\mathbf m=(m_1,\dots,m_n)=\bar D_{2,2}$.
 Then Theorem \ref{thm:jackson} gives a polynomial $p_2$ such that
 \[\|g-p_2\|_\infty= \Cjackson\sum_{j\in \J_1\cap\J_2}\frac{\lip_jg}{m_j}\leq \Cjackson|\J_1\cap\J_2|\lip (f_2)\frac2{D_{2,2}}%
 \leq \frac\varepsilon2-\eta.\]
 Also,%
 \begin{equation}\label{eq:fulldeg}
  \fulldeg p_2\leq \mathbf m=\bar D_{2,2}.%
 \end{equation}
 Let 
 \[h_1\coloneqq f_1+p_2\qquad\textrm{and}\qquad h_2\coloneqq f_2-p_2\]
 so that $f=h_1+h_2$, $h_1\geq\eta$ and $h_2\geq \eta$ on $\domain$, and $h_j\in \polysvar{\J_j}{}$. %
 The bound \eqref{eq:entrywise} follows from the definition of $h_j$ and \eqref{eq:fulldeg}. Observe also that, by the last part of Theorem \ref{thm:jackson},
 \[\lip p_2\leq 2\lip g\leq 2\lip(f_2).  \]
 Finally, we have 
 \begin{equation}\label{eq:normhj}
  \|p_2\|_\infty\leq \|g\|_\infty+\frac\varepsilon2-\eta \leq \|f_2\|_\infty+\varepsilon-\eta\leq 2\|f_2\|_\infty,
 \end{equation}
 so
 \begin{equation*}
      \|h_j\|_\infty\leq\|f_j\|_\infty+\|p_2\|_\infty\leq \|f_j\|_\infty+2\|f_2\|_\infty\leq 3(\|f_1\|_\infty+\|f_2\|_\infty).
 \end{equation*}

 For the induction step,  let $\ell\geq 3$ and set $\tilde f=f_1+\dots+f_{\ell-1}-(\ell-2)(\epsilon-\eta)$, so that we have $f-(\ell-2)(\epsilon-\eta)=\tilde f+f_\ell \geq \epsilon$ since $f\geq \varepsilon=(\ell-1)\epsilon-(\ell-2)\eta$.
  The proof for the case $\ell=2$ with $\varepsilon=\varepsilon_{\ell-1}$ gives a polynomial $p_\ell\in\polysvar{\interset_\ell}{}$ such that
 \[ \tilde f%
    -p_\ell\geq  \eta%
 \quad\textrm{and}\quad f_\ell+p_\ell\geq \eta\]
 on $\domain$, and
 with $\fulldeg p_\ell=\bar D_{\ell,\ell}$, $\lip p_\ell\leq 2\lip f_\ell$.and, analogously to \eqref{eq:normhj},
 \begin{equation}\label{eq:norminduction}
     \|p_\ell\|_\infty\leq 2\|f_\ell\|_\infty.
 \end{equation}
Write 
 \[f_1'+\dots+f'_{\ell-1}=f_1+\dots+f_{\ell-1}-p_\ell\] 
 where $f_j'=f_j-p_\ell$ for the largest $j$ with $\interset_\ell\subset \J_j$ (which must happen for some $j$, by the running intersection property; see Definition \ref{def:rip}) and $f_k'=f_k$ for all other $k\neq j$. Thus $f'_j\in \polysvar{\J_j}{}$,
 \begin{equation}\label{eq:lipdeg}
  \fulldeg f_j'\leq \max(\fulldeg f_j,\fulldeg p_\ell)=\max(\fulldeg f_j,\bar D_{\ell,\ell}), \qquad \lip f_j'\leq \lip f_j+\lip p_\ell,
\end{equation}

 The induction hypothesis applies to the polynomial
 \begin{equation*}
     f_1'+\dots+f_{\ell-1}'+(\ell-2)(\epsilon-\eta)=\tilde f+(\ell-2)(\epsilon-\eta)-p_\ell\geq (\ell-2)\epsilon-(\ell-3)\eta.
 \end{equation*}
 This means that there are polynomials $h_1,\dots,h_{\ell-1}$ %
 such that
 \begin{itemize}
  \item $f_1'+\dots+f_{\ell-1}'=f_1+\dots+f_{\ell-1}-p_\ell=h_1+\dots+h_{\ell-1}$,
  \item $h_j\in \polysvar{\J_j}{}$ for all $1\leq j\leq \ell-1$,
  \item $h_j\geq \eta$ for all $1\leq j\leq \ell-1$,
  \item We have, for all $1\leq j\leq \ell-1$,
   \begin{align*}
    \fulldeg h_j&
    \leq \max(\fulldeg f'_j,\bar D_{j,\ell}\dots,\bar D_{\ell-1,\ell})\\
    &\leq  \max(\fulldeg f_j,\fulldeg p_\ell,\bar D_{j,\ell},\dots,\bar D_{\ell-1,\ell})\\
    &= \max(\fulldeg f_j,\bar D_{j,\ell},\dots,\bar D_{\ell,\ell}).
    \end{align*}
    Observe that the second index in each $\bar D_{k,\ell}$ is $\ell$ because of the accumulation of Lipschitz constants resulting from the estimate  \eqref{eq:lipdeg}.
  \item  We have, for all $1\leq j\leq \ell-1$, again because of \eqref{eq:lipdeg},
   \[
    \lip h_j\leq 3\sum_{k=j}^\ell \lip f_k.
   \]
   \item We have, for all $1\leq j\leq \ell-1$, using \eqref{eq:norminduction}, 
    \[ \|h_j\|_\infty \leq 3\times 2^{\ell-2}\sum_{k=1}^{\ell-1} \|f'_k\|_\infty\leq 3\times 2^{\ell-2} \left(\sum_{k=1}^{\ell-1} \|f_k\|_\infty+\|p_\ell\|_\infty\right)\leq 3\times 2^{\ell-1}\sum_{k=1}^\ell\|f_k\|_\infty.\]
 \end{itemize}
 Let $h_\ell=f_\ell+p_\ell$. Then again $f_1+\dots+f_\ell=h_1+\dots+h_\ell$, $h_\ell\in\polysvar{\J_\ell}{}$, $h_\ell\geq \eta$ on $\domain$, $\fulldeg h_j\leq \max(\fulldeg f_\ell,\bar D_{\ell,\ell})$, $\lip h_\ell\leq \lip f_\ell+\lip p_\ell\leq 3\lip f_\ell$,
 \[\|h_\ell\|_\infty\leq \|f_\ell\|_\infty+\|p_\ell\|\leq
3\|f_\ell\|\leq 3\times 2^{\ell-1} \sum_{j=1}^\ell \|f_j\|_\infty,\]
 so the lemma is proven.
\end{proof}

 \section{Proofs}
 \label{sec:proofs}
 
 \subsection{Proof of Theorem {\ref{thm:mainabrv}}}
 \label{sec:proof}

 Theorem \ref{thm:mainabrv} will follow from Theorem \ref{thm:main}, which presents a more detailed bound, together with the definitions of $\overline L,M,\overline J$.

\begin{theorem}\label{thm:main}
  Let $n>0$ and $\ell\geq 2$, and let $\bfr_1,\bfr_2,\dots,\bfr_\ell\in \N^n$, $\bfr_j=(r_{j,1},\dots,r_{j,n})$, 
   be nowhere-vanishing multi-indices. Let also $\J_1,\dots,\J_\ell$ be subsets of $\{1,\dots,n\}$ satisfying the running intersection property.
  Let $p=p_1+p_2+\dots+p_\ell$ be a polynomial that is the sum of finitely many polynomials $p_j\in \polysvar{\J_j}{\bfr_j}$. Then if $p\geq \varepsilon$ on $[-1,1]^n$, we have 
  \[p\in \preordering{\bfr_1}{\J_1}+\dots+\preordering{\bfr_\ell}{\J_\ell}\] 
  as long as, for all $1\leq j\leq \ell$ and all $1\leq i\leq n$,
   \begin{align}  
   (r_{j,k}+2)^2 &\geq\frac{ 2^{\frac{|J_j|}2+2}(\ell+2)\|p\|_\infty n\pi^2}{\varepsilon}\notag\\
   &\qquad\cdot\prod_{1\leq m\leq n}\left(\max\left[(\fulldeg p_j)_m,\max_{\substack{j\leq l\leq \ell\\ m\in \interset_l}}\frac{4\Cjackson(\ell+2)|\interset_l|\sum_{t=l}^\ell\lip p_t}{\varepsilon}
    \right]+2\right)\notag\\
    \label{eq:assumptionepsilon}
    &\qquad\cdot\max_{l\in \J_j}\left[(\fulldeg p_j)_l,\max_{\substack{j\leq q\leq \ell\\ l\in \interset_q}}\frac{4\Cjackson(\ell+2)|\interset_q|\sum_{t=q}^\ell\lip p_t}{\varepsilon}
    \right]^{2} \,,
    \end{align}
 and
 \begin{equation}\label{eq:equivcond}  
  {(r_{j,i}+2)^2}\geq 2\pi^2n\max\left[\max_{1\leq m\leq n}(\fulldeg p_j)_m,\max_{j\leq k\leq \ell}\frac{4\Cjackson(\ell+2)|\interset_k|\sum_{t=k}^\ell\lip p_t}{\varepsilon}
   \right]^2 \,,
 \end{equation}
\end{theorem}

 \begin{proof}[Proof of Theorem \ref{thm:main}] 
  Let 
   \begin{equation}\label{eq:choiceeta}
    \epsilon= \frac{\varepsilon+(\ell-2)\eta}\ell\qquad\textrm{and}\qquad\eta=\frac{\varepsilon}{2(\ell+2)}
   \end{equation}
  and apply Lemma \ref{lem:approx} (with $\mathbf g=0$, so that $\domain=[-1,1]^n$) to get polynomials $h_1,\dots,h_\ell$ with $h_j\in \polysvar{J_j}{}$, $p=h_1+\dots+h_\ell$ and, for all $1\leq j\leq \ell$, $h_j\geq \eta$ on $[-1,1]^n$ and
  \[\fulldeg h_j\leq \max(\fulldeg p_j, \bar D_j,\bar D_{j+1},\dots,\bar D_\ell),\]
  where $\bar D_l:=(\delta_{j\in \interset_l}D_l)_{j=1}^\ell$, $\delta_{j\in \interset_l}$ equals $1$ if $j\in \interset_l$ and $0$ otherwise, and
  \[D_l:=\left\lceil\frac{2\Cjackson\,|\interset_l|\sum_{t=l}^\ell\lip p_t}{\varepsilon_{l-1}-2\eta}\right\rceil=\left\lceil \frac{4\Cjackson\,\ell\,|\interset_l|\sum_{t=l}^\ell\lip p_t}{\eta(\ell+2)}\right\rceil,\qquad l=2,3,\dots,\ell,\]
  (since $\varepsilon_j-2\eta=\eta(\ell+2)/2\ell$) and $D_1:=D_2$. Thus $|\bar D_l|=|\interset_l|D_l$ for $2\leq l\leq \ell$.

  Apply Corollary \ref{coro:puttingtogether} to each of the polynomials
  \[H_j=\frac{h_j-\min_{[-1,1]^n}h_j}{\max_{[-1,1]^n}h_j-\min_{[-1,1]^n}h_j}\]
  to see that, for

  \begin{equation}\label{eq:lowerbound}
   \varepsilon_j\geq 2n\pi^2\left(\prod_{1\leq k\leq n}\left((\fulldeg H_j)_k+1\right)\right)\max_{\multiindex\in\indexset{H_j}}\left(2^{\hamming(I)/2}\max_{1\leq k\leq n}\frac{i_k^2}{(r_{j,k}+2)^2}\right),
  \end{equation}
  (recall that $\indexset{H_j}$ is the set of multiindices $I=(i_1,
  \dots,i_n)$ corresponding to exponents of $x_1,
  \dots,x_n$ in the terms appearing in $H_j$ and $\hamming(I)$ is the number of nonzero entries in $I$)
    we have 
    \begin{equation}\label{eq:belong}
     H_j+\epsilon_j\in \preordering{\bfr_j}{\J_j};
    \end{equation}
    when applying the corollary, note that \eqref{eq:equivcond} {implies} \eqref{eq:effcond} in this case {because, if $I=(i_1,\dots,i_n)\in \indexset{H_j}$, then
  \begin{multline*}
    i_k\leq (\fulldeg H_j)_k\leq (\fulldeg h_j)_k \leq\max(\fulldeg p_j,\bar D_j,\dots,\bar D_\ell)\\
    \leq \max\left[\max_{1\leq m\leq n}(\fulldeg p_j)_m,\max_{j\leq k\leq \ell}\frac{4\Cjackson(\ell+2)|\interset_k|\sum_{t=k}^\ell\lip p_t}{\varepsilon}\right],
  \end{multline*}
  by the definition of $\bar D_l$.}
  Observe that \eqref{eq:belong} means also that
  \begin{equation}\label{eq:histhere}
   h_j-\min_{[-1,1]^n}h_j+\epsilon_j\left(\max_{[-1,1]^n}h_j-\min_{[-1,1]^n}h_j\right)\in\preordering{\bfr_j}{\J_j}.
  \end{equation}
  Note that we have $\fulldeg H_j=\fulldeg h_j$, $\indexset{H_j}\setminus\indexset{h_j}=\emptyset$, and  $\indexset{h_j}\setminus\indexset{H_j}\subseteq\{(0,\dots,0)\}$ since the powers of all terms in $h_j$ and in $H_j$ are the same, with the only possible exception of the constant term, which may appear in one of these and vanish in the other.
  Now, going back to our choice \eqref{eq:choiceeta} of $\eta$ and using \eqref{eq:assumptionepsilon}, we have %
  \begin{align*}
   \eta&=\frac{\varepsilon}{2(\ell+2)} \\
   &\geq \|p\|_\infty n\pi^2 \left(\max\left[\max_{1\leq m\leq n}(\fulldeg p_j)_m,\max_{j\leq k\leq \ell}\frac{2\Cjackson\,|\interset_k|\sum_{t=k}^\ell\lip p_t}{\eta}
   \right]+2\right)^{|J_j|+2}\\
   &\qquad \cdot
   \frac{2^{\frac{|J_j|}2+1}}{\min_{1\leq k\leq n}(r_{j,k}+2)^2}\\   
   &\geq \|p\|_\infty n\pi^2\left(\max\left[\max_{1\leq m\leq n}(\fulldeg p_j)_m,\max_{j\leq k\leq \ell}\left\lceil\frac{2\Cjackson\,\ell\,|\interset_k|\sum_{t=k}^\ell\lip p_t}{\eta(\ell+2)}\right\rceil
   \right]+1\right)^{|J_j|+2}\\
   &\qquad \cdot
   \frac{2^{\frac{|J_j|}2+1}}{\min_{1\leq k\leq n}(r_{j,k}+2)^2}.
   \end{align*}
  for all $j \in \{1,\ldots,\ell\}$. Notice that after separating two of the $|J_j|+2$ terms in the product and removing the $+1$ factor from them, we obtain
   \begin{align*}
   \eta &\geq \|p\|_\infty n\pi^2\left(\max\left[\max_{1\leq m\leq n}(\fulldeg p_j)_m,\max_{j\leq k\leq \ell}\left\lceil\frac{2\Cjackson\,\ell\,|\interset_k|\sum_{t=k}^\ell\lip p_t}{\eta(\ell+2)}\right\rceil
   \right]+1\right)^{|J_j|}
   2^{\frac{|J_j|}2+1}\\
   &\qquad\cdot\frac{\max\left[\max_{1\leq m\leq n}(\fulldeg p_j)_m,\max_{j\leq k\leq \ell}\left\lceil\frac{2\Cjackson\,\ell\,|\interset_k|\sum_{t=k}^\ell\lip p_t}{\eta(\ell+2)}\right\rceil
   \right]^2}{\min_{1\leq k\leq n}(r_{j,k}+2)^2}\\
   &\geq \|p\|_\infty n\pi^2\left(%
   \prod_{k\in \J_j}(\left[\max\left(\fulldeg p_j,\bar D_j,\dots,\bar D_\ell\right)\right]_k+1)\right)%
   2^{\frac{|J_j|}2+1}\\
   &\qquad\cdot\frac{\max_{%
   1\leq k\leq n%
   }\left[\max\left(\fulldeg p_j,\bar D_j,\dots,\bar D_\ell\right)\right]_k^2}{\min_{1\leq k\leq n}(r_{j,k}+2)^2} ,
   \end{align*}
   where we have used the definition of $\bar D_l$, as well as the fact that each factor has been replaced by one that is smaller or equal, the original expression containing the maximum of them on each factor. Next, use  $\fulldeg H_j\leq \max(\fulldeg p_j,\bar D_j,\dots, \bar D_\ell)$ as well as $\hamming(I)\leq |J_j|$ for every multi-index $I$ in $\indexset{H_j}$, which is true because $H_j\in\polysvar{J_j}{}$, yielding
   \begin{align*}
   \eta &\geq \|p\|_\infty \left(2n\pi^2\left(\prod_{1\leq k\leq n}\left(\left(\fulldeg H_j\right)_k+1\right)\right)\max_{\multiindex\in\indexset{H_j}}\left(2^{\hamming(I)/2}\max_{1\leq k\leq n}\frac{i_k^2}{(r_{j,k}+2)^2}\right)\right)\\
   &\geq \left(\max_{[-1,1]^n}h_j-\min_{[-1,1]^n}h_j\right)\\
   &\qquad \cdot\left(2n\pi^2\left(\prod_{1\leq k\leq n}\left((\fulldeg H_j)_k+1\right)\right)\max_{\multiindex\in\indexset{H_j}}\left(2^{\hamming(I)/2}\max_{1\leq k\leq n}\frac{i_k^2}{(r_{j,k}+2)^2}\right)\right),
  \end{align*}
since we have $\max_{[-1,1]^n}h_j-\min_{[-1,1]^n}h_j\leq \|p\|_\infty$.
 With this bound for $\eta$, together with the fact that $\min_{[-1,1]^n}h_j\geq \eta$, we get 
 \begin{align*}
  h_j&\geq h_j-\min_{[-1,1]^n}h_j+\eta\\
  &\geq h_j-\min_{[-1,1]^n}h_j +\left(\max_{[-1,1]^n}h_j-\min_{[-1,1]^n}h_j\right)\\
  &\qquad \cdot \left(2n\pi^2\left(\prod_{1\leq k\leq n}\left((\fulldeg H_j)_k+1\right)\right)\max_{\multiindex\in\indexset{H_j}}\left(2^{\hamming(I)/2}\max_{1\leq k\leq n}\frac{i_k^2}{(r_{j,k}+2)^2}\right)\right),
 \end{align*}
 so that, by \eqref{eq:lowerbound} and \eqref{eq:histhere}, $h_j\in \preordering{\bfr}{\J_j}$ and hence 
 \[p=h_1+\dots+h_\ell\in \preordering{\bfr}{\J_1}+\dots+\preordering{\bfr}{\J_\ell}.\qedhere\]
 \end{proof}

\subsection{{Proof of Theorem \ref{thm:putinar}}}
\label{sec:putinarproof}

\emph{Overview.} For this proof, we will first use the sparse approximation theory developed in Section \ref{sec:approximation} to represent the sparse polynomial $p$ as a sum of positive polynomials $h_1+\dots+h_{\ell}$, each of them depending on a clique of variables $J_j$. We then  work with each of these polynomials $h_j$ using the tools developed by Baldi--Mourrain \cite{baldi2021moment} to write $h_j=\f_j+\hat q_j$, where  $\hat q_j$ is by construction obviously an element of the corresponding quadratic module, and $\f_j$ is strictly positive on $[-1,1]^n$. Thus Corollary \ref{coro:puttingtogether} can be applied to $\f_j$, which shows that it belongs to the preordering, and then one argues (also following the ideas of \cite{baldi2021moment}) that the preordering is contained in the quadratic module, hence giving that $\f_j$ is contained in the latter as well. In sum, this shows that $h_j$ is in the quadratic module, which is what want. Most of the heavy lifting goes to estimating the minimum of $\f_j$ to justify the application of Corollary \ref{coro:puttingtogether}. 

\begin{proof}[Proof of Theorem \ref{thm:putinar}]
For each $j=1,\dots,
 \ell$, pick $\Call>0$  such that the following two bounds are satisfied:
  \begin{align}\label{eq:Callcond1}
  \Call &\geq 2\pi^2|J_j|^{1+\frac{16\lojL_j}3}\Cd^2 \Cjackson^{\frac{16\lojL_j}3}2^{1+2(4+3\frac{8}3)\lojL_j}3^{\frac{(16+8\ell)\lojL_j+2}3}\cardg^{-\frac23}\lojC_j^{\frac83}(\max_{i\in K_j}\deg g_i)^2(2(\ell+2))^{8\lojL_j},\\
 \Call&\geq \Cf(\Cjackson\Cm)^{(2\lojL_j+|J_j|+2)(1+\frac{8\lojL_j}3)} |J_j|\pi^22^{
 4\lojL_j+\frac{|J_j|}2+1+(1+\frac{4\lojL_j+1}3)(2\lojL_j+|J_j|+2)(1+\frac{8\lojL_j}3)
  }\nonumber\\
  &\qquad\times 3^{
   \ell(\lojL_j+1)+(2\lojL_j+|J_j|+2)(1+\frac{8\lojL_j}3)
  }(\ell+2)^{1+\lojL_j+\frac{4\lojL_j+1}3(2\lojL_j+|J_j|+2)(1+\frac{8\lojL_j}3)}\cardg\nonumber\\
  &\qquad\times \lojC_j^{ 1+ \frac34(2\lojL_j + |J_j|+2)(1+\frac{8\lojL_j}3)} \left(\sum_{i=j}^\ell|\interset_i|^{2(2\lojL_j+|J_j|+2)(1+\frac{8\lojL_j}3)}\right)\nonumber\\
& \qquad \cdot (\max_{k\in K_j}\deg g_k+1)^{(2\lojL_j+|J_j|+2)(1+\frac{8\lojL_j}3)}.\label{eq:Callcond2}
  \end{align}
  Note that these only depend on $\mathbf g$ and  $J_1,\dots,J_\ell$.

Apply Lemma \ref{lem:approx} to $f=p$, $f_i=p_i$, $\epsilon=3\varepsilon/2\ell$, $\eta=\varepsilon/2(\ell+2)$ to get polynomials $h_1,\dots, h_\ell$ such that 
\begin{equation}\label{eq:starstar}
p=h_1+\dots+h_\ell,\qquad h_i\in \polysvar{\J_i}{},\qquad h_i(x)\geq \eta=\frac{\varepsilon}{2(\ell+2)} \;\text{for}\;x\in\domain,
\end{equation}
and 
\begin{equation}\label{eq:degh}
 \fulldeg h_i\leq \max(\fulldeg p_i,\bar D_{i,\ell},\dots,\bar D_{\ell,\ell})_{J_i}. 
\end{equation}

 In the dense setting, Baldi--Mourrain \cite{baldi2021moment} construct a family of single-variable polynomials \[(\bmh_{t,m})_{(t,m)\in \N\times\N}\] providing useful approximation properties that we have adapted to the (separated-variables) sparse setting and collected in Lemma \ref{lem:bmproperties}. 
 To state this, we set, for all $j=1,\dots,\ell$ and for $(t_j,m_j)\in\N\times \N$ as well as for $s_j>0$,
\begin{gather}
 \label{eq:defq} q_{j, t_j, m_j}(x):=\sum_{i\in K_j}\bmh_{t_i,m_i}\left(g_i(x)\right)^2g_i(x),\\
 \label{eq:deff} f_{j,s_j, t_j, m_j}(x):=h_j(x)-s_j \,q_{j, t_j,m_j}(x).
\end{gather}
Let us give an idea of what these functions do. The single-variable polynomial $\bmh_{t_j,m_j}$ is of degree $m_j$ and roughly speaking approximates the function that equals 1 on $(-\infty,0)$ and $1/{t_j}$ elsewhere.  Thus $q_{j,t_j,m_j}$ almost vanishes (for large $t_j$) on $S(\mathbf g_{K_j})$, and outside of this domain it is roughly a sum of multiples of the negative parts of $\mathbf g_{K_j}$'s  entries.
The definition of $f_{j,s_j,t_j,m_j}$ is engineered to obtain a polynomial that is almost equal to $h_j$ in $S(\mathbf g_{K_j})$ yet remains positive throughout $[-1,1]^n$. Instead of going into the details of the construction, we record the properties we need in the following lemma.
\begin{lemma}[{a version of \cite[Props. 2.13, 3.1, and 3.2, Lem. 3.5]{baldi2021moment}}]\label{lem:bmproperties}
 Assume \eqref{eq:gsizeassumption} and the Archimedean conditions \eqref{eq:archimedeanity} are satisfied. Then for each $j=1,\dots,\ell$ there are values $s_j,t_j, m_j$ of the parameters involved in Definition \eqref{eq:defq} and Definition \eqref{eq:deff}, such that the following holds with the shorthands
 \begin{equation}\label{eq:fhat}
     \f_j=f_{j,s_j,t_j,m_j}\qquad\text{and}\qquad \hat q_j=s_jq_{j,t_j,m_j}:
 \end{equation}
 \begin{enumerate}[ref=(\roman*),label=\roman*.]
  \item \label{it:minf} \cite[Prop.~3.1]{baldi2021moment} gives 
  \[\f_j(x)\geq \frac12\min_{y\in S(\mathbf g_{K_j})}h_j(y)\geq \frac{\eta}2=\frac{\varepsilon}{4(\ell+2)}\quad\text{for all}\quad x\in [1,1]^n.\]
  \item\label{it:degQ} We have $\hat q_j\in \quadraticmoduleg{\mathbf r}{\mathbf g}{J_j}$ for all multi-indices $\mathbf r=(r_1,\dots,r_n)$ with
  \[r_i\geq (2m_{j}+ 1)\max_{k\in K_j}(\fulldeg g_k)_i.\]
  \item \label{it:asymptm} \cite[eq.~$(20)$]{baldi2021moment} gives the existence of a constant $\Cm>0$ such that %
   \[ m_{j}\leq \Cm\lojC_j^{\frac43}\cardg^{\frac13}2^{4\lojL_j}( 
   \deg h_j%
   )^{\frac{8\lojL_j}3}\left(\frac{\min_{x\in \domain}h_j(x)}{\|h_j\|_\infty}\right)^{-\frac{4\lojL_j+1}3}.\]
  \item \label{it:sizef} \cite[eq.~$(16)$]{baldi2021moment} gives the existence of a constant $\Cf>0$ such that
  \[\|\f_{j}\|_\infty\leq \Cf\|h_j\|_\infty2^{3\lojL_j}\cardg \lojC_j(\deg h_j)^{2\lojL_j}\left(\frac{\min_{x\in \domain}h_j(x)}{\|h_j\|_\infty}\right)^{-\lojL_j}.\]
  \item \label{it:degf} \cite[eq.~$(17)$]{baldi2021moment} gives the existence of a constant $\Cd>0$ such that 
  \[\deg \f_{j}\leq \Cd2^{4\lojL_j}\cardg^{\frac13}\lojC_j^{\frac43}(\max_{i\in K_j}\deg g_i)(\deg h_j)^{\frac{8\lojL_j}3}\left(\frac{\min_{x\in \domain}h_j(x)}{\|h_j\|_\infty}\right)^{-\frac{4\lojL_j+1}3}.\]
 \end{enumerate}
\end{lemma}

 Item \ref{it:degQ} follows
 \footnote{This calculation is slightly different to the one in \cite[Lem. 3.5]{baldi2021moment} because the definition of $q_{j,t_j,m_j}$ (or in their notations, $f-p$) differs from the one given there in that the functions $\bmh_j$ are squared here, an idea we take from the exposition of the results of \cite{baldi2021moment} in the dissertation of L. Baldi 
 and that is advantageous because then $q_{j,t_j,m_j}\in \quadraticmoduleg{\bfr}{\mathbf g}{J_j}$ automatically. This requires taking $m_j$ twice as large, and we absorb this difference into the constant $\Cm$.} from $\deg \bmh_{t_j,m_j}=m_j$ and the definition of $q_{j,t_j,m_j}$.
 The proofs of the other items can be found in the indicated sources.

 Take $s_j,t_j,m_j,\f_j, \hat q_j$ for $j=1,\dots,\ell$ satisfying the properties (i)-(v) collected in Lemma \ref{lem:bmproperties}. 
Continuing with the proof of Theorem \ref{thm:putinar}, denote 
 \begin{equation*}
  F_j\coloneqq\frac{\f_{j}-\min_{[-1,1]^n}   \f_{j}}{\max_{[-1,1]^n} \f_{j}-\min_{[-1,1]^n} \f_{j}}.
  \end{equation*}
Since $\f_{j}\geq \varepsilon/4(\ell+2)$ on $[-1,1]^n$, we may apply Corollary \ref{coro:puttingtogether} with $p=F_{j}$ to get that 
  \begin{equation}\label{eq:Fjbelongs}
  F_j+\epsilon_j\in \preordering{\bfr_j}{\J_j}
  \end{equation}
 as long as 
  \begin{equation}\label{eq:belongingcondition}
   \epsilon_j\geq 2|J_j|\pi^2\left(\prod_{i\in J_j}((\fulldeg F_j)_i+1)\right)\max_{I=(i_1,\dots,i_n)\in \indexset{F_j}}\left(2^{\hamming(I)/2}\max_{1\leq k\leq n}\frac{i_k^2}{(r_{j,k}+2)^2}\right)
  \end{equation}
 and \eqref{eq:effcond} are verified. 
 In this context, the condition \eqref{eq:effcond} required in Corollary \ref{coro:puttingtogether} is equivalent to the theorem's assumption \eqref{eq:putinarcond2}; let us show how this works: First, using $i_k\leq (\fulldeg F_j)_k$, $\fulldeg F_j\leq \fulldeg \f_j\leq (\deg \f_j)\ones$, Lemma \ref{lem:bmproperties}\ref{it:degf}, we get
 \begin{align*}
  2\pi^2 |J_j|i_k^2
  &\leq 2\pi^2|J_j|\left(\fulldeg F_j\right)^2_k\leq 2\pi^2|J_j|(\deg \f_j)^2\\
   &\leq 2\pi^2|J_j|\left(\Cd 2^{4\lojL_j}\cardg^{\frac13}\lojC_j^{\frac43}(\max_{i\in K_j}\deg g_i)(\deg h_j)^{\frac{8\lojL_j}3}\left(\frac{\min_{x\in S(\mathbf g)}h_j(x)}{\|h_j\|_\infty}\right)^{-\frac{4\lojL_j+1}{3}}\right)^2.
   \end{align*}
   Now use equation  \eqref{eq:degh}  to get that this is
   \begin{align*}
   &\leq 2\pi^2|J_j| \\
   & \qquad \cdot \left(\Cd 2^{4\lojL_j}\cardg^{\frac13}\lojC_j^{\frac43}(\max_{i\in K_j}\deg g_i)\max\left(\fulldeg p_j,\bar D_{j,\ell},\dots,\bar D_{\ell,\ell}\right)^{\frac{8\lojL_j}3}\left(\frac{\min_{x\in S(\mathbf g)}h_j(x)}{\|h_j\|_\infty}\right)^{-\frac{4\lojL_j+1}{3}}\right)^2
   \\
  &\leq 2\pi^2|J_j|\left(\Cd 2^{4\lojL_j}\cardg^{\frac13}\lojC_j^{\frac43}(\max_{i\in K_j}\deg g_i)\max\left(\fulldeg p_j,\bar D_{j,\ell},\dots,\bar D_{\ell,\ell}\right)^{\frac{8\lojL_j}3}\left(\frac{3^\ell\sum_i\|p_i\|_\infty%
  }{\varepsilon/2(\ell+2)}\right)^{\frac{4\lojL_j+1}3}\right)^2,
  \end{align*}
  where we have also used  the fact that 
  \begin{equation}\label{eq:boundhj}\|h_j\|_\infty\leq3\times 2^{\ell-1}\sum_{i=1}^\ell\|p_i\|_\infty\leq 3^\ell\sum_{i=1}^\ell\|p_i\|_\infty
  \end{equation}
  and the last estimate from  \eqref{eq:starstar}.
  Next, use \eqref{eq:defD},
   $|\interset_j|\leq |J_j|$, $\varepsilon_i-2\eta_i=\frac{\ell+6}{2\ell(\ell+2)}\varepsilon$ to get
  \begin{align*}
 2\pi^2|J_j|i^2_k &\leq 2\pi|J_j|%
  \left( \Cd2^{4\lojL_j}\cardg^{-\frac13}\lojC_j^{\frac43}(\max_{i\in K_j}g_i)(\deg p_j)^{\frac{8\lojL_j}3}\left(\frac{2\Cjackson|J_j|(3\sum_{i=1}^\ell\lip p_i)}{\frac{\ell+6}{2\ell(\ell+2)}\varepsilon}\right)^{\frac{8\lojL_j}3}\right. \\
  &\qquad \cdot \left. \left(\frac{
  3^\ell\sum_i\|p_i\|_\infty%
  }{\varepsilon/2(\ell+2)}\right)^{\frac{4\lojL_j+1}3}\right)^2
  \\
  &\leq \Call\left(\left(\sum_{i=1}^\ell\|p_i\|_\infty%
  \right)^{\frac{4\lojL_j+1}3}(\deg p_j)^{\frac{8\lojL_j}3}\frac{\left(\sum_{i=1}^{\ell}\lip p_i\right)^{\frac{8\lojL_j}3}}{\varepsilon^{\frac{(8+4)\lojL_j+1}3}}\right)^2\leq (r_{j,k}+2)^2,
 \end{align*}
 where we have additionally used  equation \eqref{eq:Callcond1} and our assumption \eqref{eq:putinarcond2}; this is precisely \eqref{eq:effcond}.

 We would next like to show that 
 \begin{equation} \label{eq:fhatpreordering}
  \f_{j}\in \preordering{\bfr_j}{\J_j}.
 \end{equation} 
 Let us first explain why this will be enough to prove the theorem.
 Once we have \eqref{eq:fhatpreordering}, by Lemma \ref{lem:archimedean} $\f_j$ is also contained in $\quadraticmoduleg{\mathbf r_j+\mathbf 2}{1-\sum_{i\in J_j}x_i^2}{J_j}$, and it is our assumption \eqref{eq:archimedeanity} that $\quadraticmoduleg{\mathbf r_j+\mathbf 2}{1-\sum_{i\in J_j}x_i^2}{J_j}\subseteq \quadraticmoduleg{\bfr_j+\mathbf 2}{\mathbf g_{K_j}}{J_j}$. In other words, we have
  \[ \f_j\in\quadraticmoduleg{\bfr_j+\mathbf 2}{\mathbf g_{K_j}}{J_j}.\]
  By Lemma \ref{lem:bmproperties}\ref{it:degQ}, $\hat q_{j}$ also belongs to $\quadraticmoduleg{\bfr_j+\mathbf 2}{\mathbf g_{K_j}}{J_j}$, so we can conclude that 
  \[h_j\in \quadraticmoduleg{\bfr_j+\mathbf 2}{\mathbf g_{K_j}}{J_j},\]
  which is equivalent to the conclusion of the theorem.
  
 Thus we need to prove \eqref{eq:fhatpreordering}. Let us explain why, in order to obtain this conclusion %
 we just need to show that
 \begin{multline}\label{eq:needtoprove}
  \frac{\eta}2=\frac{\varepsilon}{4(\ell+2)} \geq (\max_{[-1,1]^n} \f_{j}-\min_{[-1,1]^n} \f_{j}) \;2|J_j|\pi^2\left(\prod_{i\in J_j}((\fulldeg F_j)_i+1)\right)\\
  \times\max_{I=(i_1,\dots,i_n)\in \indexset{F_j}}\left(2^{\hamming(I)/2}\max_{1\leq k\leq n}\frac{i_k^2}{(r_{j,k}+2)^2}\right).
 \end{multline}
 Observe that \eqref{eq:Fjbelongs} is equivalent to 
 \begin{equation}\label{eq:equivform}
  \f_j-\min_{[-1,1]^n}\f_j+\epsilon_j(\max_{[-1,1]^n}\f_j-\min_{[-1,1]^n}\f_j)\in \preordering{\bfr_j }{J_j}
 \end{equation}
 If \eqref{eq:needtoprove} were true, we would then have
\begin{align*}
  \f_j&\geq \f_j-\min_{[-1,1]^n}\f_j+\frac{\varepsilon}{4(\ell+2)}\\
  &\geq \f_j-\min_{[-1,1]^n}\f_j
  +  (\max_{[-1,1]^n} \f_{j}-\min_{[-1,1]^n} \f_{j}) \;2|J_j|\pi^2\left(\prod_{i\in J_j}((\fulldeg F_j)_i+1)\right)
  \\
  &\hspace{4cm}\times
  \max_{I=(i_1,\dots,i_n)\in \indexset{F_j}}\left(2^{\hamming(I)/2}\max_{1\leq k\leq n}\frac{i_k^2}{(r_{j,k}+2)^2}\right).
\end{align*}
So in view of \eqref{eq:belongingcondition} and \eqref{eq:equivform}, we would indeed have $\f_j\in \preordering{\bfr_j}{J_j}$, which is~\eqref{eq:fhatpreordering}.

Let us now collect some preliminary estimates that will help us to prove \eqref{eq:needtoprove}.  
  For $I\in \indexset{F_j}$ we have $\hamming(I)\leq |J_j|$ so we  estimate
 \begin{equation}\label{eq:hammingestimate}
  2^{\hamming(I)/2}\leq 2^{|J_j|/2}.
 \end{equation}
 We  also estimate 
 \begin{equation}\label{eq:maxisqestimate}
  \frac{i^2_k}{(r_{j,k}+2)^2}\leq \frac{\left(\fulldeg \f_{j}\right)^2_k}{\min_{1\leq l\leq n}(r_{j,l}+2)^2}.
 \end{equation}

 Now we will estimate $ \max_{[-1,1]^n} \f_{j}-\min_{[-1,1]^n} \f_{j}$ from above.
 Using Lemma \ref{lem:bmproperties}\ref{it:minf} and \ref{it:sizef}, we get 
 \begin{align}
  \notag \max_{[-1,1]^n} \f_{j}&-\min_{[-1,1]^n} \f_{j}\leq \Cf\|h_j\|_\infty2^{3\lojL_j}\cardg\lojC_j(\deg h_j)^{2\lojL_j}\left(\frac{\min_{x\in \domain}h_j(x)}{\|h_j\|_\infty}\right)^{-\lojL_j}-\frac{\varepsilon}{4(\ell+2)}
  \end{align}
 Use  \eqref{eq:starstar}, \eqref{eq:degh} and \eqref{eq:boundhj} to see that this is
  \begin{align}
  \notag  \max_{[-1,1]^n} \f_{j}&-\min_{[-1,1]^n} \f_{j}\\
  &\leq \Cf\left(3^\ell\sum_i\|p_i\|_\infty \right)2^{3\lojL_j}\cardg\lojC_j\max(\deg p_j,|\bar D_{j,\ell}|,\dots,|\bar D_{\ell,\ell}|)^{2\lojL_j}\notag\\
  & \qquad \cdot \left(\frac{\varepsilon}{2(\ell+2)3^\ell\sum_i\|p_i\|_\infty}\right)^{-\lojL_j}\notag \\
   &= \Cf\left(3^\ell\sum_i\|p_i\|_\infty \right)2^{3\lojL_j}\cardg\lojC_j\max(\deg p_j,|\interset_j| D_{j,\ell},\dots,|\interset_\ell| D_{\ell,\ell})^{2\lojL_j}\notag\\
   & \qquad \cdot \left(\frac{\varepsilon}{2(\ell+2)3^\ell\sum_i\|p_i\|_\infty}\right)^{-\lojL_j}
  \label{eq:uppersum}
 \end{align}
 For the last line, we have used the definition of $\bar D_{l,m}$ as in Lemma \ref{lem:approx}.

 Additionally, we obtain the following estimate %
 \begin{align}
 \fulldeg \f_{j}\leq  &\notag\max\left(\fulldeg h_j,\fulldeg \hat q_{j}%
 \right)\\
  \notag  &\leq \max\left(\fulldeg p_j,\bar D_{j,\ell},\dots,\bar D_{\ell,\ell},(2m_j+1)\max_{k\in K_j}\fulldeg g_k\right)\\
  \notag  &\leq \max\Big(\fulldeg p_j,\bar D_{j,\ell},\dots,\bar D_{\ell,\ell},\\
   \notag &\qquad(2\left(\Cm\lojC_j^{\frac34}\cardg^{-\frac13}2^{4\lojL_j}(
   \deg h_j%
   )^{\frac{8\lojL_j}3}\left(\frac{\min_{x\in \domain}h_j(x)}{\|h_j\|_\infty}\right)^{-\frac{4\lojL_j+1}3}\right)+1)\max_{k\in K_j}\fulldeg g_k\Big) \\
  \notag& \leq \max\Big(\fulldeg p_j,\bar D_{j,\ell},\dots,\bar D_{\ell,\ell},\\
  &\qquad \left(\Cm 2^{4\lojL_j+1}\lojC^{\frac34}\cardg^{-\frac13}(\deg h_j)^{\frac{8\lojL_j}{3}}\left(\frac{3^\ell\sum_i\|p_i\|_\infty}{\varepsilon/2(\ell+2)}\right)^{\frac{4\lojL_j+1}3}+1\right)\max_{k\in K_j} \fulldeg g_k\Big).
  \label{eq:firstdegestimate} 
 \end{align}
The first inequality comes from  \eqref{eq:deff}, 
the second one from \eqref{eq:degh} and Lemma \ref{lem:bmproperties}\ref{it:degQ}, 
the third one from Lemma 
\ref{lem:bmproperties}\ref{it:asymptm}, 
and the last one from \eqref{eq:starstar}, \eqref{eq:boundhj}, and \eqref{eq:degh}.
Compare with Lemma \ref{lem:bmproperties}\ref{it:degf}.

With those estimates under our belt, we now turn to showing that \eqref{eq:needtoprove} is true.  Using \eqref{eq:uppersum}, \eqref{eq:hammingestimate}, \eqref{eq:maxisqestimate}, as well as $\fulldeg F_j\leq \fulldeg \f_j$, we can start to estimate the right-hand side of \eqref{eq:needtoprove} by
 \begin{align*}
  &(\max_{[-1,1]^n} \f_{j}-\min_{[-1,1]^n} \f_{j}) \;2|J_j|\pi^2\left(\prod_{i\in J_j}((\fulldeg F_j)_i+1)\right)\\
  &\qquad\qquad\qquad\qquad\qquad\times\max_{I=(i_1,\dots,i_n)\in \indexset{F_j}}\left(2^{\hamming(I)/2}\max_{1\leq k\leq n}\frac{i_k^2}{(r_{j,k}+2)^2}\right) \\
  &\leq \Cf\left(3^\ell\sum_i\|p_i\|_\infty \right)2^{3\lojL_j}\cardg\lojC_j\max(\deg p_j,|\interset_j| D_{j,\ell},\dots,|\interset_\ell| D_{\ell,\ell})^{2\lojL_j}\left(\frac{\varepsilon}{2(\ell+2)3^\ell\sum_i\|p_i\|_\infty}\right)^{-\lojL_j}\\
  &\qquad\qquad\qquad \times 2|J_j|\pi^2\left(\max_{1\leq i\leq n}(\fulldeg \f_{j})_i+1\right)^{|J_j|}\frac{2^{|J_j|/2}\max_{1\leq i\leq n}\left(\fulldeg \f_{j}\right)^2_{i}}{\min_{1\leq k\leq n}(r_{j,k}+2)^2}.
  \end{align*}
   Next we  denote
   \[m_{j,\ell}=\max(\deg p_j,|\interset_j|D_{j,\ell},\dots,|\interset_\ell|D_{\ell,\ell})\]
   and we reorganize and consolidate the terms and then we use \eqref{eq:firstdegestimate} to see that this is
  \begin{align*}
    &\leq\varepsilon^{-\lojL_j} \Cf |J_j|\pi^2\left(3^\ell\sum_i\|p_i\|_\infty \right)^{\lojL_j+1}2^{4\lojL_j+\frac{|J_j|}2+1}(\ell+2)^{\lojL_j}\cardg\lojC_jm_{j,\ell}^{2\lojL_j}\\%
  &\qquad\qquad\qquad \times \left(\max_{1\leq i\leq n}(\fulldeg f_{j})_i+1\right)^{|J_j|+2}\frac{1}{\min_{1\leq k\leq n}(r_{j,k}+2)^2}\\
  &\leq \varepsilon^{-\lojL_j} \Cf |J_j|\pi^2\left(3^\ell\sum_i\|p_i\|_\infty \right)^{\lojL_j+1}2^{4\lojL_j+\frac{|J_j|}2+1}(\ell+2)^{\lojL_j}\cardg\lojC_jm_{j,\ell}^{2\lojL_j}\\%
  &\qquad\qquad\qquad \times \Big(\max\Big(\fulldeg p_j,\bar D_{j,\ell},\dots,\bar D_{\ell,\ell},\\
  &\qquad\qquad\qquad\qquad\quad\left(\Cm 2^{4\lojL_j+1}\lojC^{\frac34}\cardg^{-\frac13}
  (
  \deg h_j%
  )^{\frac{8\lojL_j}{3}}\left(\frac{3^\ell\sum_i\|p_i\|_\infty}{\varepsilon/2(\ell+2)}\right)^{\frac{4\lojL_j+1}3}+1\right)\\
  &\qquad\qquad\qquad\qquad\times\max_{k\in K_j}\fulldeg g_k\Big) +1\Big)^{|J_j|+2}\frac{1}{\min_{1\leq k\leq n}(r_{j,k}+2)^2}\\
    &\leq \varepsilon^{-\lojL_j} \Cf |J_j|\pi^2\left(3^\ell\sum_i\|p_i\|_\infty \right)^{\lojL_j+1}2^{4\lojL_j+\frac{|J_j|}2+1}(\ell+2)^{\lojL_j}\cardg\lojC_jm_{j,\ell}^{2\lojL_j}\\%
  &\qquad \times \Big(\max\Big(\fulldeg p_j,\bar D_{j,\ell},\dots,\bar D_{\ell,\ell},\\
  &\qquad\qquad\quad\left(\Cm 2^{4\lojL_j+1}\lojC^{\frac34}\cardg^{-\frac13}
  m_{j,\ell}^{\frac{8\lojL_j}{3}}\left(\frac{3^\ell\sum_i\|p_i\|_\infty}{\varepsilon/2(\ell+2)}\right)^{\frac{4\lojL_j+1}3}+1\right)\\
  &\qquad\qquad\times\max_{k\in K_j}\fulldeg g_k\Big) +1\Big)^{|J_j|+2}\frac{1}{\min_{1\leq k\leq n}(r_{j,k}+2)^2} \,.
 \end{align*}
 Now use \eqref{eq:defD} as well as $\varepsilon_i-2\eta_i=\frac{\ell+6}{2\ell(\ell+2)}\varepsilon$ to see that the above is bounded by
 \begin{align*}
  &\varepsilon^{-\lojL_j} \Cf |J_j|\pi^2\left(3^\ell\sum_i\|p_i\|_\infty \right)^{\lojL_j+1}2^{4\lojL_j+\frac{|J_j|}2+1}(\ell+2)^{\lojL_j}\cardg\lojC_j\\
  &\qquad\times\Big(\max\Big[\deg p_j,|\interset_j|^2 \frac{2\ell(\ell+2)}{\ell+6}\frac{2\Cjackson\sum_{k=j}^\ell\lip \f_k}{\varepsilon}+1,\dots,\\
  &\qquad\qquad\qquad\qquad|\interset_\ell|^2 \frac{2\ell(\ell+2)}{\ell+6}\frac{2\Cjackson\lip \f_\ell}{\varepsilon}+1,\\%
  &\qquad\qquad\qquad\qquad\quad\left(\Cm 2^{4\lojL_j+1}\lojC^{\frac34}\cardg^{-\frac13}%
  \left(\frac{3^\ell\sum_i\|p_i\|_\infty}{\varepsilon/2(\ell+2)}\right)^{\frac{4\lojL_j+1}3}+1\right)\max_{k\in K_j}\fulldeg g_k\Big]\\
  &\qquad\qquad\qquad\qquad+1\Big)^{(2\lojL_j+|J_j|+2)(1+\frac{8\lojL_j}3)}\frac{1}{\min_{1\leq k\leq n}(r_{j,k}+2)^2}\\
  &\leq \varepsilon^{-\lojL_j-\frac{4\lojL_j+1}3(2\lojL_j+|J_j|+2)(1+\frac{8\lojL_j}3)} \Cf(\Cjackson\Cm)^{(2\lojL_j+|J_j|+2)(1+\frac{8\lojL_j}3)} |J_j|\pi^2\left(3^\ell\sum_i\|p_i\|_\infty \right)^{\lojL_j+1}\\
  &\qquad \times 2^{4\lojL_j+\frac{|J_j|}2+1+(1+\frac{4\lojL_j+1}3)(2\lojL_j+|J_j|+2)(1+\frac{8\lojL_j}3)
  }(\ell+2)^{1+\lojL_j+\frac{4\lojL_j+1}3(2\lojL_j+|J_j|+2)(1+\frac{8\lojL_j}3)}\cardg\\
  &\qquad\times\lojC_j^{1+\frac34(2\lojL_j+|J_j|+2)(1+\frac{8\lojL_j}3)}\left(\sum_{i=j}^\ell|\interset_i|^{2(2\lojL_j+|J_j|+2)(1+\frac{8\lojL_j}3)}\right)\\
  &\qquad\times (3\deg p_j\sum_{i=1}^\ell\lip p_i)^{(2\lojL_j+|J_j|+2)(1+\frac{8\lojL_j}3)}\\
  &\qquad \times%
  (\max_{k\in K_j}\deg g_k+1)^{(2\lojL_j+|J_j|+2)(1+\frac{8\lojL_j}3)}\frac{1}{\min_{1\leq k\leq n}(r_{j,k}+2)^2}.
  \end{align*}
  Finally use  \eqref{eq:Callcond2} and then  \eqref{eq:putinarcond1} to get that the above is less  than
  \begin{multline*}
  \Call \frac{\varepsilon^{-\lojL_j-\frac{4\lojL_j+1}3(2\lojL_j+|J_j|+2)(1+\frac{8\lojL_j}3)}\left(\sum_i\|p_i\|_\infty \right)^{\lojL_j+1}(\deg p_j\sum_{i=1}^\ell\lip p_i)^{(2\lojL_j+|J_j|+2)(1+\frac{8\lojL_j}3)}}{\min_{1\leq k\leq n}(r_{j,k}+2)^2}\\
  \leq \frac{\varepsilon}{4(\ell+2)}.
 \end{multline*}
  This shows that \eqref{eq:needtoprove} holds, and hence also \eqref{eq:fhatpreordering}, which proves the theorem.
\end{proof}

\begin{lemma}[{\cite[Lemma 3.8]{baldi2021moment}}]\label{lem:archimedean}
 Let $\J\subset\{1,\dots,n\}$, and let $\bfr=(r_1,\dots,r_n)$ be a multi-index such that $r_i>0$ only if $i\in \J$.
 The quadratic module $\quadraticmoduleg{\bfr+\twos}{1-\sum_{i\in \J}x_i^2}{J}$ contains the preordering $\preordering{\bfr}{\J}$,
 \[\preordering{\bfr}{\J}\subseteq \quadraticmoduleg{\bfr+\twos}{1-\textstyle\sum_{i\in \J}x_i^2}{J}.\]
\end{lemma}
\begin{proof}
This follows from
\[1\pm x_i=\frac12(1-x_i^2+(1\pm x_i)^2)=\frac12((1-\|x\|^2)+\sum_{\substack{j\in\J\\j\neq i}}x_j^2+(1\pm x_i)^2)\]
and
\begin{align*}
 1-x_i^2&=(1-x_i)(1+x_i)\\
 &=\frac14(1-\|x\|_\J^2)\left(2\sum_{j\neq i}x_j^2+(1-x_i)^2+(1+x_i)^2\right)\\
 &\qquad +\frac14(1-\|x\|_\J^2)^2+\left(\sum_{j\neq i}x_j^2\right)^2+\sum_{j\neq i}x_j^2((1-x_i)^2+(1+x_i)^2).
\end{align*}
The increase of $\twos$ in $\bfr$ stems from the fact that $\deg(1-x_i^2)=2$ while the degree of the right-hand side above is 4.
\end{proof}

\subsection{An asymptotic lemma}

\begin{lemma}\label{lem:binomasympt}
 For $a,b,c,d,p,q>0$, with $cq-ap\neq 0$,
 \[\limsup_{\varepsilon\searrow0}\Bigg|\frac{\displaystyle\log\left[\binom{a+b\varepsilon^{-p}}{b\varepsilon^{-p}}\middle/\binom{c+d\varepsilon^{-q}}{d\varepsilon^{-q}}\right]}{(cq-ap)\log\varepsilon}\Bigg|<+\infty.
\]
\end{lemma}
\begin{proof}
 Using the Stirling approximation, which states that there exist $C'\geq 0$ such that
 \[|\log(n!)-n\log n+n|\leq C'\log n\]
 for all large $n>0$, we get that,
\begin{align}
     \notag &\limsup_{\varepsilon\searrow0}\left|\frac{\log\left[\binom{a+b\varepsilon^{-p}}{b\varepsilon^{-p}}\middle/\binom{c+d\varepsilon^{-q}}{d\varepsilon^{-q}}\right]}{(cq-ap)\log\varepsilon}\right|\\
    \notag  &\quad=\limsup_{\varepsilon\searrow0}\left|\frac{\log\left[\frac{(a+b\varepsilon^{-p})!}{a!(b\varepsilon^{-p})!}\middle/\frac{(c+d\varepsilon^{-q})!}{c!(d\varepsilon^{-q})!}\right]}{(cq-ap)\log\varepsilon}\right|\\
     \notag &\quad=\limsup_{\varepsilon\searrow0}\left|\frac{\log(a+b\varepsilon^{-p})!-\log a!-\log (b\varepsilon^{-p})!-\log(c+d\varepsilon^{-q})!+\log c!+\log(d\varepsilon^{-q})!}{(cq-ap)\log\varepsilon}\right|\\
     \notag &\quad=\limsup_{\varepsilon\searrow0}\left|\frac{\log(a+b\varepsilon^{-p})!-\log (b\varepsilon^{-p})!-\log(c+d\varepsilon^{-q})!+\log(d\varepsilon^{-q})!}{(cq-ap)\log\varepsilon}\right|\\
     \label{eq:importantpart1}&\quad=\limsup_{\varepsilon\searrow0}\frac{1}{|cq-ap||\log\varepsilon|}\Big|(a+b\varepsilon^{-p})\log(a+b\varepsilon^{-p})%
 -b\varepsilon^{-p}\log(b\varepsilon^{-p})\\
   \label{eq:importantpart2}  &\quad\qquad-(c+d\varepsilon^{-q})\log(c+d\varepsilon^{-q})+d\varepsilon^{-q}\log(d\varepsilon^{-q})\\
    \label{eq:unimportantpart1}   &\quad\qquad-\log a!+\log c!-(a+b\varepsilon^{-p})+(b\varepsilon^{-p})+(c+d\varepsilon^{-q})-(d\varepsilon^{-q}))\\
     \label{eq:unimportantpart2} &\quad\qquad+O\big(|\log(a+b\varepsilon^{-p})|+|\log(b\varepsilon^{-p})|+|\log(c+d\varepsilon^{-q})|+|\log(d\varepsilon^{-q})|\big)\Big|
\end{align}
Notice that the terms in \eqref{eq:unimportantpart1} are asymptotically much smaller than the denominator,  which tends to $+\infty$. 
The last line \eqref{eq:unimportantpart2} is, as $\varepsilon\searrow0$, of the order of
\begin{align*}
&\frac{O(|\log(a+b\varepsilon^{-p})|+|\log(b\varepsilon^{-p})|+|\log(c+d\varepsilon^{-q})|+|\log(d\varepsilon^{-q})|)}{|cq-ap||\log \varepsilon|}\\
&=\frac{O(2|\log b|+2p|\log \varepsilon|+2|\log d|+2q|\log \varepsilon|)}{|cq-ap||\log \varepsilon|}\\
&=\frac{(p+q)O(|\log \varepsilon|)}{|cq-ap||\log \varepsilon|}=O\left(\frac{p+q}{|cq-ap|}\right)=O(1).
\end{align*}
Let us now show that the remaining two lines \eqref{eq:importantpart1}--\eqref{eq:importantpart2} tend to 1 in the limit.
Now, 
\[\lim_{\varepsilon\searrow0} b\varepsilon^{-p}\left(\log(a+b\varepsilon^{-p})-\log(b\varepsilon^{-p})\right)=a\quad\textrm{and}\quad
\lim_{\varepsilon\searrow0} d\varepsilon^{-p}\left(\log(c+d\varepsilon^{-p})-\log(d\varepsilon^{-p})\right)=c,\]
so we get 
\begin{align*}
&\lim_{\varepsilon\searrow0}\frac{1}{(cq-ap)\log\varepsilon}\Big((a+b\varepsilon^{-p})\log(a+b\varepsilon^{-p})-a\log a-b\varepsilon^{-p}\log(b\varepsilon^{-p})\\
    &\quad\qquad-(c+d\varepsilon^{-q})\log(c+d\varepsilon^{-q})+c\log c+d\varepsilon^{-q}\log(d\varepsilon^{-q})\Big)\\
    &\quad=
    \lim_{\varepsilon\searrow0}\frac{-a\log a+c\log c+a-c+a\log(a+b\varepsilon^{-p})-c\log(c+d\varepsilon^{-q})}{(cq-ap)\log\varepsilon}\\
    &\quad=\lim_{\varepsilon\searrow0}\frac{a\log(a+b\varepsilon^{-p})-c\log(c+d\varepsilon^{-q})}{(cq-ap)\log\varepsilon}.
\end{align*}
In this quotient, both the numerator and the denominator tend to $\pm\infty$, so we can apply a version of the l'H\^opital rule, which states that, if the limit of the quotient of their derivatives exists, then the original limit above equals that limit. Taking the limit of the quotient of the derivatives gives
\begin{align*}
    &\lim_{\varepsilon\searrow0}\frac{\frac{cdq\varepsilon^{-1-q}}{c+d\varepsilon^{-q}}-\frac{abp\varepsilon^{-1-p}}{a+b\varepsilon^{-p}}}{(cq-ap)/\varepsilon}\\
    &\quad=\lim_{\varepsilon\searrow0}
        \frac{cdq\varepsilon^{-q}(a+b\varepsilon^{-p})-abp\varepsilon^{-p}(c+d\varepsilon^{-q})}{(cq-ap)(a+b\varepsilon^{-p})(c+d\varepsilon^{-q})}\\
    &\quad=\lim_{\varepsilon\searrow0}
        \frac{cdq(a\varepsilon^{p}+b)-abp(c\varepsilon^{q}+d)}{(cq-ap)(a\varepsilon^{q}+b)(c\varepsilon^{p}+d)}\\    
    &\quad=\lim_{\varepsilon\searrow0}
        \frac{cdqb-abpd}{(cq-ap)bd}=1.    \qedhere
\end{align*}
\end{proof}

\end{document}